\documentclass[11pt]{amsart}
\usepackage{amsmath}
\usepackage{amsfonts}
\usepackage{amsthm}
\usepackage[a4paper,bindingoffset=0.2in,%
left=1in,right=1in,top=1in,bottom=1in,%
footskip=.25in]{geometry}
\usepackage[english]{babel}

\usepackage[backend=bibtex, sorting=nyt, style=alphabetic]{biblatex}
\usepackage{titlesec}
\usepackage{graphicx}
\usepackage{breqn}
\usepackage{xcolor}

\makeatletter
\let\NAT@parse\undefined
\makeatother
\usepackage{hyperref}

\theoremstyle{plain}
\newtheorem{thm}{Theorem}
\newtheorem{lem}{Lemma}
\newtheorem{prop}{Proposition}
\newtheorem{coro}{Corollary}

\theoremstyle{definition}
\newtheorem{defn}{Definition}

\titleformat{\section}
{\normalfont\scshape}{\thesection}{1em}{}

\begin{document}
	\title{The supercritical deformed Hermitian Yang--Mills equation on compact projective manifolds}
	\author{Aashirwad Ballal}
	\email{aashirwadb@iisc.ac.in}
	\maketitle
	\begin{abstract}
		In this paper, we extend a result of \cite{gao} regarding the solvability of the twisted deformed Hermitian Yang-Mills equations on compact K\"ahler manifolds to allow for the twisting function to be non-constant and slightly negative in all dimensions. Using this result along with the methods in \cite{datpin20}, we prove that the twisted dHYM equation on compact, projective manifolds can be solved provided certain numerical conditions are satisfied. As a corollary, we obtain a new proof in the projective case of a recent theorem of \cite{taka} addressing a conjecture of \cite{cjy}.
	\end{abstract}
	 
	\section{Introduction}
	
	Let $ M  $ be a compact, connected K\"ahler manifold of complex dimension $ n $ with a K\"ahler form  $  \chi  $ and let $  \omega_0 $ be a smooth, real $ (1,1)$-form on $ M $. Any other such $  \omega  $ in the same cohomology class as $  \omega_0 $ can be written as $ \omega =  \omega_0  +   \sqrt{-1} \partial \bar{\partial} \phi $ for some smooth function $ \phi $ and we will denote by $ \lambda_1(x), \dots, \lambda_n(x) $ the eigenvalues of $  \omega  $ with respect to the metric $ \chi $ at a point $ x \in M $. The deformed Hermitian Yang-Mills equation can then be written as 
	\begin{equation} \label{eq:dhyme}
		\sum_{i = 1}^{n} \text{arccot}(\lambda_i(x)) = \theta
	\end{equation}
	where $ \theta $ is a real constant called the ``phase angle" and the range of $ \text{arccot} $ is chosen to be $ (0, \pi) $. The connections of the dHYM equation to string theory and mirror symmetry are discussed in \cite{cxy}. Necessarily, we must have $ \theta \in (0, n \pi) $. When $ \theta \in (0, \pi) $ the phase is said to be supercritical and this is the case we consider in this paper. \\

	When the phase is supercritical, it can be shown that (see the next section) solving the dHYM equation is equivalent to finding an $ \omega \in [\omega_0] $ such that 
	\begin{equation}
	 \label{leveln}
	 \Re(\omega +  \sqrt{-1} \chi )^n - \cot(\theta) \Im(\omega +  \sqrt{-1} \chi)^{n} = 0
	\end{equation}and
	\begin{equation}
	\label{levelsubn}
	 \Re(\omega +  \sqrt{-1} \chi )^k - \cot(\theta) \Im(\omega +  \sqrt{-1} \chi)^k > 0
	\end{equation}where $ k = 0, 1, \dots, n-1 $. In fact, the positivity condition \ref{levelsubn} alone is sufficient, assuming the integral of the left side of equation \ref{leveln} is zero (\cite{cjy}). More generally, \cite{gao} proved the following result:
	
	\begin{prop}{\cite[Proposition~5.5]{gao}}
		\label{prop:gctwisted}
		Let $ (M, \chi) $ be a compact K\"ahler manifold, $ 0 < \theta < \Theta < \pi $ be constants, and $ f $ a smooth function on $ M $. Assume that there exists a smooth $  \omega_0 $ satisfying $ \omega_0 \in \Gamma_{\chi, \theta, \Theta} $ (for this notation, see the next section). Suppose that
		\begin{equation*}
			\int_M \Re(\omega_0 +  \sqrt{-1} \chi )^n - \cot(\theta) \Im(\omega_0 +  \sqrt{-1} \chi)^{n} = \int_M f \chi^{n} \geq 0
		\end{equation*}
		Then there exists a constant $ \epsilon > 0 $ depending only on $ n, \theta, \Theta $ such that if 
		\begin{equation*}
			\begin{split}
			 n \geq 4  \text{ and }  f > -\epsilon , \text{ or} \\
			 n \leq 3  \text{ and }  f \geq 0  \text{ is constant} \\
			\end{split}
		\end{equation*} there exists a smooth $ \omega \in [\omega_0] $ with $ \omega \in \Gamma_{\chi, \theta, \Theta} $ and 
		\begin{equation*}
		 \Re(\omega +  \sqrt{-1} \chi )^n - \cot(\theta) \Im(\omega +  \sqrt{-1} \chi)^{n} = f \chi^{n}
		\end{equation*}
	\end{prop}

	The ``twisted" dHYM equation is 
	\begin{equation*}
		\Re(\omega +  \sqrt{-1} \chi )^n - \cot(\theta) \Im(\omega +  \sqrt{-1} \chi)^{n} = f \chi^{n} 
	\end{equation*}When the dimension $ n $ of the K\"ahler manifold is greater than $ 3 $, the above proposition gives conditions for the solvability of the twisted dHYM equation when $ f $ is allowed to be ``slightly negative" whereas in dimensions less than $ 4 $, $ f $ is assumed to be non-negative and constant. The first main result of this paper is to extend this theorem to allow a slightly negative and non-constant $ f $ when $ n < 4 $, and in particular, for $ n = 3 $.
	
	\begin{thm}[Twisted dHYM in dimensions $ < $ 4]
		\label{mainpdethm}
		Let $ (M, \chi)$ be a compact K\"ahler manifold of dimension $ n $, $ 0 < \theta < \pi $ a constant, $ f $ a smooth function on $ M $, and $ \omega_0 $ a smooth form in $ \Gamma_{\chi, \theta} $. There exists a constant $ \epsilon > 0 $ depending only on $ n, \theta $ such that if the following conditions hold:
		\begin{equation*}
		\begin{split}
		\int_M \Re(\omega_0 +  \sqrt{-1} \chi )^n - \cot(\theta) \Im(\omega_0 +  \sqrt{-1} \chi)^{n} = \int_M f \chi^{n} \geq 0 \\
		f > -\epsilon
		\end{split}		
		\end{equation*}
		then there exists a smooth $ \omega \in [\omega_0] $ with $ \omega \in \Gamma_{\chi, \theta} $ and 
		\begin{equation*}
		\Re(\omega +  \sqrt{-1} \chi )^n - \cot(\theta) \Im(\omega +  \sqrt{-1} \chi)^{n} = f \chi^{n}
		\end{equation*}
		Further, if $\Theta \in (\theta,  \pi) $ is another constant then $ \epsilon > 0 $ can be chosen (depending also on $ \Theta $ now) such that if the above conditions are satisfied, then the solution $ \omega $ is also in $ \Gamma_{\chi, \theta, \Theta} $.
	\end{thm}
	
	Taking $ f = 0 $ in Proposition \ref{prop:gctwisted} above shows the equivalence of solving the dHYM equation with the existence of a smooth form $ \omega_1 \in [\omega_0]  $ satisfying certain point-wise positivity conditions. However, one can do better, and in fact, it was also shown in \cite{gao} that these point-wise positivity conditions are also equivalent to certain seemingly weaker ``numerical" conditions similar to uniform versions of the ones in \cite[Theorem~4.2]{dempaun}. Recently, these numerical conditions were further weakened to their non-uniform versions in \cite{taka} using the methods in \cite{song} and \cite{gao} and consequently \cite[Conjecture~1.4]{cjy} was proved in the projective case. In this paper, we will provide a proof of an analogue of this conjecture in the projective case for the twisted dHYM equation using the methods of \cite{datpin20} and \cite{gao}. To do this, we first prove 
	\begin{thm}
		\label{maindhymthm}
		Suppose $ M $, $ \chi $, $ \omega_0 $ are as above. Assume that $ M $ is projective, $ \theta \in (0, \pi) $,
		\begin{equation*}
			\int_M \Re(\omega_0 +  \sqrt{-1} \chi )^n - \cot(\theta) \Im(\omega_0 +  \sqrt{-1} \chi)^{n} \geq 0
		\end{equation*} and
		\begin{equation*}
			\int_V \Re(\omega_0 +  \sqrt{-1} \chi )^k - \cot(\theta) \Im(\omega_0 +  \sqrt{-1} \chi)^{k} > 0
		\end{equation*}for every $ k $-dimensional subvariety $ V $ of $ M $ where $ k = 0, 1, \dots, n - 1 $. Then there exists a smooth $ \omega \in [\omega_0] $ with $ \omega \in \Gamma_{\chi, \theta} $.
	\end{thm}
	
	From Theorem \ref{maindhymthm},  Theorem \ref{mainpdethm} (for $ n < 4 $) and \cite[Proposition~5.2]{gao} (for $ n \geq 4 $), we immediately see that the twisted dHYM equation on projective manifolds has a solution provided the numerical conditions in Theorem \ref{maindhymthm} are satisfied:
	
	\begin{coro}
		Let $ (M, \chi)$ be a compact projective manifold of dimension $ n $ and let $ 0 < \theta < \Theta < \pi $ be a constants. There exists a constant $ \epsilon > 0 $ depending only on $ n, \theta, \Theta $ such that if $ f $ is a smooth function on $ M $ with $ f > -\epsilon $, and $ \omega_0 $ is a smooth, real $ (1, 1) $-form on $ M $ satisfying
		\begin{equation*}
		\int_M \Re(\omega_0 +  \sqrt{-1} \chi )^n - \cot(\theta) \Im(\omega_0 +  \sqrt{-1} \chi)^{n} = \int_M f \chi^{n} \geq 0
		\end{equation*}
		\begin{equation*}
			\int_V \Re(\omega_0 +  \sqrt{-1} \chi )^k - \cot(\theta) \Im(\omega_0 +  \sqrt{-1} \chi)^{k} > 0
		\end{equation*}for every $ k $-dimensional subvariety $ V $ of $ M $ where $ k = 0, 1, \dots, n - 1 $ \\
		then there exists a smooth $ \omega \in [\omega_0] $ with $ \omega \in \Gamma_{\chi, \theta, \Theta} $ and 
		\begin{equation*}
		\Re(\omega +  \sqrt{-1} \chi )^n - \cot(\theta) \Im(\omega +  \sqrt{-1} \chi)^{n} = f \chi^{n}.
		\end{equation*}
	\end{coro}
	
	We now sketch an outline of the approaches to the above theorems. \\
	
	The $n = 1 $ case of Theorem \ref{mainpdethm} follows immediately from the fact that $ H^{(1,1)}(M, \mathbb{R})  = \mathbb{R}$. For $ n = 2 $, it again follows almost immediately from Yau's solution \cite{yaucalab} of the Calabi conjecture. The case of most interest to us here is hence $ n = 3 $. Here, we use a method of continuity as in \cite{pin3}. The new complications in the twisted dHYM equation as opposed to the non-twisted one will mainly show up while proving the Laplacian estimates. 
	\\
	
	To prove Theorem \ref{maindhymthm}, we proceed as in \cite{datpin20}. Using \cite[Proposition~5.2]{gao} and the PDE results for the twisted dHYM equation (Proposition \ref{prop:gctwisted} and Theorem \ref{mainpdethm}), we prove by a concentration of mass technique that there exists a positive $(1, 1)$-current $ \Theta \geq 2\beta [Y], \Theta \in [\Omega_0] $ satisfying the cone condition with respect to $ \chi $. Here, $ Y $ is the zero set of a section of a very ample line bundle on $ M $, so there exists a K\"ahler metric $ \chi_Y $ in the same cohomology class as $ [Y] $. Hence, we can add the exact current $ \displaystyle \beta \chi_Y - \beta [Y] $ to $ \Theta $ to obtain a K\"ahler current $ T \geq \beta[Y] $ satisfying the cone condition on $ M \cap Y^{c} $. We then produce a smooth metric in a neighborhood of $ Y $ using induction, a degenerate concentration of mass and a few successive regularization arguments. This smooth metric is then glued together with the regularizations of $ T $ to get a smooth metric $ \Omega \in [\Omega_0] $ satisfying the cone condition on $ M $.  \\
	
	\textbf{Acknowledgments}. The author would like to thank Dr. Vamsi Pingali for suggesting this problem, for various helpful discussions and also for several improvements to the first draft of this paper. This work was supported by a scholarship from the Indian Institute of Science.

\section{The cone conditions}

In this section we prove and quote some results regarding the dHYM equation which will be used throughout the paper. As before, let $ (M, \chi) $ be a K\"ahler manifold of dimension $ n $. A smooth, real $ (1, 1) $-form $ \omega $ is said to satisfy the deformed Hermitian Yang-Mills equation on $ (M, \chi) $ with phase angle $ \theta $ if 
\begin{equation}
	\label{eq:eigendhym}
	\sum_{i = 1}^n \text{arccot}(\lambda_i) = \theta
\end{equation}where $ (\lambda_i)_i $ are the eigenvalues of $ \omega $ with respect to the K\"ahler form $ \chi $. \\

Throughout the paper, we assume that the phase angle is supercritical i.e. $ \theta \in (0, \pi) $. In this case, equation $ \ref{eq:eigendhym} $ is equivalent to the following:
\begin{equation}
	\label{eq:formdhym}
	\Re(\omega +  \sqrt{-1} \chi)^n - \cot(\theta) \Im(\omega +  \sqrt{-1} \chi)^n = 0 
\end{equation}\begin{align*}
	\Re(\omega +  \sqrt{-1} \chi)^k - \cot(\theta) \Im(\omega +  \sqrt{-1} \chi)^k > 0 && k = 1, \dots, n - 1
\end{align*}
(the positivity of forms above is as defined in \cite{dembook}. The notions of strong positivity and positivity in \cite{dembook} coincide when we deal with polynomial expressions of two real $ (1, 1) $-forms, of which is K\"ahler, as in \ref{eq:formdhym}). \\

To see the equivalence of \ref{eq:eigendhym} and \ref{eq:formdhym} (when the phase is supercritical), choose a point $ x \in M $ and a basis $ (\xi_i)_i $ of $ \bigwedge^{(1, 0)} T^*_x $ such that $ \chi(x) =  \sqrt{-1} \sum_{ i } \xi_i \wedge \bar{\xi_i} $ and $ \omega(x) = \sqrt{-1} \sum_{ i } \lambda_i(x) \xi_i \wedge \bar{\xi_i} $. In this basis, 
\begin{align}
\label{eq:diagonalized}
 \Re((\omega +  \sqrt{-1} \chi)^k) - \cot(\theta) \Im((\omega +  \sqrt{-1} \chi)^k)(x) = && \text{}  \\  \sum_{K \subset \{ 1, \dots, n \}, |K|=k} k!\sqrt{-1}^{k} \prod_{i \in K}\sqrt{1 + \lambda_i(x)^2} \frac{\sin(\theta - \sum_{i \in K} \theta_i(x))}{\sin(\theta)} \bigwedge_{i \in K} \xi_i \wedge \bar{\xi_i} 
\end{align}
  where $ \theta_i(x) := \text{arccot}(\lambda_i(x)) $. So the required equivalence follows immediately from the following elementary lemma:
\begin{lem}
	\label{sinelem}
	Let $ \theta, \theta_1, \dots, \theta_k  \in (0, \pi)$, then $ \displaystyle \sum_{i = 1}^k \theta_i < \theta $ if and only if $ \displaystyle \sin(\theta - \sum_{ i \in I}\theta_i) > 0 $ for all subsets $ I $ of $ \{1, \dots, k\} $. 
\end{lem}
\begin{proof}
	If $ \displaystyle \sum_{i = 1}^k \theta_i < \theta $, then certainly we must have $ 0 < \displaystyle \sum_{i \in I} \theta_i < \theta $ for all non-empty subsets $ I $ of $\{1, \dots, k\}$ and consequently $ \displaystyle \sin(\theta - \sum_{ i \in I}\theta_i) > 0 $. If $ I $ is empty, then $ \sin(\theta) > 0 $ as $ \theta \in (0, \pi) $ \\
	
	For the converse, we use induction on $ k $. If $ k = 0 $, there is nothing to prove. Assume the lemma has been proved for $ k - 1  $. By the induction hypothesis, we must have $ \displaystyle 0 < \sum_{i = 1}^{k-1} \theta_i < \theta $. Hence, $ \displaystyle -\pi < \theta - \sum_{i = 1}^{k} \theta_i < \theta $. As we have $ \displaystyle \sin(\theta - \sum_{i = 1}^{k} \theta_i) > 0 $, it follows that $ \displaystyle \theta > \sum_{i = 1}^{k} \theta_i $.
\end{proof}

For any two real $ (1, 1) $-forms $ \alpha $, $ \beta $, we define the polynomial $ G_\theta^k(\alpha, \beta) $ as $$ G_\theta^k(\alpha, \beta) = \Re(\alpha +  \sqrt{-1} \beta)^k - \cot(\theta) \Im(\alpha +  \sqrt{-1} \beta)^k  $$From the discussion above, we see that for a real $ (1, 1) $-form $ \omega $ on $ M $ and $ \theta \in (0, \pi) $, the following conditions are equivalent
\begin{align}
\label{coneeq}
	\sup_{I \subset \{1, \dots, n\}, |I| = m} \sum_{i \in I} \text{arccot}(\lambda_i) < \theta & \iff \\
	G^k_\theta(\omega, \chi) > 0 \text{ for  } k = 1, \dots, m
\end{align}
For some small values of $ k $, the polynomials $ G^k_\theta $ are as below:
\begin{align*}
	G_\theta^1(\alpha, \beta) &= \alpha - \cot(\theta) \beta \\
 G_\theta^2(\alpha, \beta) &= \alpha^2 - 2\cot(\theta)\alpha \beta  - \beta^2 \\
 G_\theta^3(\alpha, \beta) &= \alpha^3 - 3\cot(\theta) \alpha^2 \chi -  3 \alpha \chi^2 + \cot(\theta) \chi^3
\end{align*}

\begin{defn}[$ \Gamma^m_{\chi, \theta}, \Gamma_{\chi, \theta}, \Gamma_{\chi, \theta, \Theta} $]
	\label{def:gamcon}
	Let $ \chi $ be a strictly positive $ (1, 1) $-form and $ \omega $ a real $ (1, 1) $-form. When any of the equivalent conditions \ref{coneeq} or  hold, we say that $ \omega \in \Gamma^m_{\chi, \theta} $. If $ m = n - 1 $, we omit the superscript in $ \Gamma_{\chi, \theta}^{n-1} $ and write $ \omega \in \Gamma_{\chi, \theta} $. If $ \Theta \in (0, \pi) $ and $ \omega \in \Gamma_{\chi, \Theta}^n \cap \Gamma_{\chi, \theta}^{n-1} $, we write $ \omega \in \Gamma_{\chi, \theta, \Theta} $
\end{defn}
Of course, these definitions also make sense if $ \omega $ is replaced by a Hermitian matrix and $ \chi $ by a positive definite matrix, so they can also be considered point-wise. \\

We see that if $ n > 1 $ and $ \omega  \in \Gamma_{\chi, \theta} $, then $ \Omega_{\theta} := G^{1}_{\theta}(\omega, \chi) = \omega - \cot(\theta) \chi > 0 $. In this paper, it will be convenient to work with the pair $ (\Omega_{\theta}, \chi) $ instead of $ (\omega, \chi) $, so we also define the polynomials $ P^k_{\theta}(\alpha, \beta) $ for real $ (1, 1) $-forms $ \alpha, \beta $ as follows:

\begin{equation*}
	P_\theta^k(\alpha, \beta) := G_\theta^k(\alpha + \cot(\theta)\beta, \beta)
\end{equation*}
In particular, $ \displaystyle P_\theta^k(\Omega_\theta, \chi) = G_\theta^k(\omega, \chi) $. For some small values of $ k $, the polynomials $ P_\theta^k $ are as below
\begin{align*}
	P_\theta^1(\alpha, \beta) &= \alpha \\
	P_\theta^2(\alpha, \beta) &= \alpha^2 - \csc^{2}(\theta) \beta^2 \\
	P_\theta^3(\alpha, \beta) &= \alpha^3 - 3 \csc^2(\theta) \alpha \beta^2 - 2 \csc^2(\theta) \cot(\theta) \beta^3  
\end{align*}
We also note the following useful equations for any real $ (1,1) $-forms $ \alpha, \beta, \delta $
\begin{align}
\label{binomsum}
	G^k_\theta(\alpha + \delta, \beta) &= \sum_{r } \binom{k}{r} G_\theta^r(\alpha, \beta) \delta^{k - r} \\
	P^k_\theta(\alpha + \delta, \beta) &= \sum_{r } \binom{k}{r} P_\theta^r(\alpha, \beta) \delta^{k - r}
\end{align}
Here, $ G^0_{\theta}, P^0_{\theta} $ denote the constant polynomial $ 1 $.
Analogous to definition $ \ref{def:gamcon} $, we make the following definition
\begin{defn}[Cone condition for smooth forms]
	\label{def:coneOmega}
	Let $ \Omega $ be a real $ (1, 1) $-form and $ \chi $ a strictly positive $ (1, 1) $-form. For $ 1 \leq m \leq n $, we say $ \Omega \in C^m_{\chi, \theta} $ if $ \displaystyle P^k_\theta(\Omega, \chi) > 0 $ for $ k = 1, \dots, m $. We also denote this as $ (\Omega, \chi) \in C^m_{\theta}$ and say that the pair $ (\Omega, \chi) $ satisfies the cone condition $  C^m_\theta $. If $ m = n - 1 $, we write $ C_{\chi, \theta} $ in place of $ C_{\chi, \theta}^{n-1} $.
\end{defn}
It is known that the sets $ \Gamma^m_{\chi, \theta} $ are convex (\cite{yuyuan}, \cite{gao}). From this it immediately follows that the sets $ C^m_{\chi, \theta}  $ are also convex. If $ (\Omega, \chi) \in C^m_\theta  $, then in particular, $ \Omega > 0 $, so if $ \chi_0 $ is another strictly positive form such that $ \chi_0 \leq \chi $, then by the min-max principle, the eigenvalues of $ \Omega $ with respect to $ \chi_0 $ are greater (not necessarily strictly) than the eigenvalues of $ \Omega $ with respect to $ \chi $ (precisely, the i-th largest eigenvalue with respect to $ \chi_0 $ is larger than the i-th largest eigenvalue with respect to $ \chi $). It is then not hard to see using the relations \ref{binomsum} that $ (\Omega, \chi_0) \in C^m_\theta $. \\

As in \cite{gao} we also define the cone condition for a current. Fix a non-negative smooth function $ \displaystyle \rho $ on $ \mathbb{R}_{\geq 0} $ supported in $ [0, 1] $ such that $ \displaystyle \int_{\mathbb{R}_{\geq 0}} \rho(t) t^{2n-1} |\mathbb{S}^{2n-1}| dt = 1 $. For any $ \delta > 0 $ and $ L^1_{loc} $ function $ \phi $ on $ \mathbb{C}^n $, we define the $ \delta $-mollification of $ \phi $ as $ \phi_{,\delta}(x) = \int_{\mathbb{C}^{n}} \delta^{-2n} \phi(x - y) \rho(\frac{|y|}{\delta}) dy $. 
\begin{defn}[Cone condition for positive currents]
	\label{def:conecurrent}
	Let $ (M, \chi) $ be a K\"ahler manifold. If $ T $ is a positive $ (1, 1) $-current, we say that the pair $ (T, \chi) $ satisfies the cone condition $ \displaystyle \overline{C_{\chi, \theta}} $ on an open set $ O $ if for every coordinate chart $ U \subset O $ where $ T|_U =  \sqrt{-1} \partial \bar{\partial} \phi_U $ and every $ \delta > 0 $, on $ \displaystyle U_\delta := {x \in U: \overline{B_\delta(x)} \subset U } $, we have $ P_\theta^k( \sqrt{-1} \partial \bar{\partial} \phi_{U, \delta}, \chi_0) \geq 0 $ for $ k = 1, \dots, m $ and $ \chi_0 $ any K\"ahler form on $ U $ with constant coefficients satisfying $ \chi_0 \leq \chi $ on $ B_\delta(x) $ (the coordinate ball of radius $ \delta $ around $ x $).
\end{defn}
This cone condition has the important property of being closed under weak limits and it can be seen that for a smooth form $ \Omega $, $ \Omega \in \overline{C^{m}_{\chi, \theta}} $ iff $ \Omega(x) $ is in the closure of $ {C^{m}_{\chi(x), \theta}} $ (hence the notation). \\

We now prove some useful results to be used later in the paper. The following lemma shows that one can ensure that one can perturb the background positive-definite matrix and still maintain the cone condition by assuming a little more positivity.
\begin{lem}
	\label{lem:perturb1}
	If $ \chi_0 $ is a positive-definite matrix, $ \omega $ is Hermitian with respect to $ \chi_0 $, $ \omega \in \Gamma_{\chi_0, \theta} $ and $ \epsilon_1 > 0 $,  then there exists an $ \epsilon_2 > 0$ depending only on $ \epsilon_1, \theta, n $ such that for all $ \chi_0 \leq \chi \leq \chi_0(1 + \epsilon_2) $, we have $ \omega + 2\epsilon_1 \chi \in \Gamma_{\chi, \theta} $
\end{lem}
\begin{proof}
	We first show that for all $ \chi \geq \chi_0 $, $ \omega + \epsilon_1 \chi \in \Gamma_{\chi_0, \theta - \epsilon_3} $  some $ \epsilon_3 $ depending only on $ \epsilon_1, \theta, n $. Let $ (\lambda_i)_{i = 1, \dots, n} $ denote the eigenvalues of $ \omega $ with respect to $ \chi_0 $ ordered so that $ \lambda_1 \leq \dots \leq \lambda_n $ and $ (\gamma_i)_{i = 1, \dots, n} $ denote the eigenvalues of $ \omega + \epsilon_1 \chi $ with respect to $ \chi_0 $ ordered similarly. It is not hard to see that we must have $ \gamma_i \geq \lambda_i + \epsilon_1 $. 
	\\
	
	Let $ I $ be any subset of $ \{1, \dots, n\} $ of $ n - 1 $ elements. If $ \displaystyle \lambda_i \geq \cot \bigg(\frac{\theta}{2(1 + n^2)} \bigg) $ $ \forall i $, then $ \displaystyle \sum_{ i \in I} \text{arccot}(\gamma_i) < \sum_{i \in I} \text{arccot}(\lambda_i) < \frac{n\theta}{2(1 + n^2)} < \frac{\theta}{2}  $. If not, then $ \displaystyle \lambda_1 \in \bigg[\cot(\theta), \cot\bigg(\frac{\theta}{2(1 + n^2)}\bigg) \bigg] $ and by the continuity and monotonicity of $ \text{arccot} $ on this compact interval, there exists an $ \epsilon_4 > 0 $ such that $ \text{arccot}(\lambda + \epsilon_1) < \text{arccot}(\lambda) - \epsilon_4 $ for all $ \lambda \in \bigg[\cot(\theta), \cot\bigg(\frac{\theta}{2(1 + n^2)}\bigg) \bigg]  $, so we must have $ \displaystyle \sum_{i \in I} \text{arccot}(\gamma_i) \leq \sum_{i \in I} \text{arccot}(\lambda_i + \epsilon) \leq \sum_{i \in I} \text{arccot}(\lambda_i) - \epsilon_4  $. Hence, we can take $ \displaystyle \epsilon_3 = \min\bigg(\epsilon_4, \frac{\theta}{2} \bigg) $. 
	\\
	
	To prove the lemma, it is now sufficient to show that $ \exists \epsilon_2 > 0 $ such that if $ \chi_0 \leq \chi \leq \chi_0(1 + \epsilon_2) $ and $ \omega \in \Gamma_{\chi_0, \theta - \epsilon_3} $, then $ \omega + \epsilon_1 \chi \in  \Gamma_{\chi, \theta} $. Let $ M $ be so large that $ \displaystyle \cot\bigg(\frac{\theta}{M} \bigg) > \epsilon_1 $ and also $ \displaystyle M > \frac{2 n \theta}{\epsilon_3} $. We define $$ \displaystyle \epsilon_2 = \min \bigg( \frac{\epsilon_1}{\cot(\frac{\theta}{M} )}, \frac{\cot(\frac{\theta}{M} )}{\cot(\frac{\theta}{M} ) - \epsilon_1} - 1 \bigg) $$ As before, let  $ (\lambda_i)_{i = 1, \dots, n} $ denote the eigenvalues of $ \omega $ with respect to $ \chi_0 $ ordered so that $ \lambda_1 \leq \dots \leq \lambda_n $ but now let $ (\gamma_i)_{i = 1, \dots, n} $ denote the eigenvalues of $ \omega $ with respect to $ \chi $ ordered similarly. The eigenvalues of $ \omega + \epsilon_1 \chi $ with respect to $ \chi $ are then $ (\gamma_i + \epsilon_1)_{i = 1, \dots, n}  $. \\
	
	If $ n = 2 $ and both $ \lambda_1 $ and $ \lambda_2 $ are non-positive, then $ \omega $ is negative semi-definite, so $ \gamma_i \geq \lambda_i $ as $ \chi \geq \chi_0 $ and the lemma follows immediately. If $ n = 2 $ and at most $ \lambda_1 \leq 0 $, then the following argument for $ n > 2 $ applies without change. If $ n > 2 $, at most $ \lambda_1 $ can be non-positive (because if not, then $ \text{arccot}(\lambda_{1}) + \text{arccot}(\lambda_2) \geq \pi > \theta $), so $ \lambda_i > 0  $ for $ i = 2, \dots, n $. By the min-max principle, it can be seen that $$ \gamma_i \geq \frac{\lambda_i}{1 + \epsilon_3} $$ for $ i = 2, \dots, n $. Also by the min-max principle, if $ \lambda_1 < 0 $, then $ \gamma_1 \geq \lambda_1 $ and if $ \lambda_1 \geq 0 $, then $  \gamma_1 \geq \frac{\lambda_1}{1 + \epsilon_3} $. 
	\\
	
	As before, let $ I $ be any subset of $ \{1, \dots, n\} $ of $ n - 1 $ elements. Now $$ \sum_{i \in I} \text{arccot}(\gamma_i + \epsilon_1) = \sum_{i \in I: \lambda_i < 0} \text{arccot}(\gamma_i + \epsilon_1) + \sum_{i \in I: 0 \leq \lambda_i \leq \cot(\frac{\theta}{M})} \text{arccot}(\gamma_i + \epsilon_1) + \sum_{i \in I: \lambda_i > \cot(\frac{\theta}{M})} \text{arccot}(\gamma_i + \epsilon_1) $$The first sum on the right can have at most one summand and as seen above, in that case we would have $ \text{arccot}(\gamma_1 + \epsilon_1) < \text{arccot}(\gamma_1) \leq \text{arccot}(\lambda_1) $ so that $$ \sum_{i \in I: \lambda_i < 0} \text{arccot}(\gamma_i + \epsilon_1) < \sum_{i \in I: \lambda_i < 0} \text{arccot}(\lambda_i) .$$
	If $ 0 \leq \lambda_i \leq \cot(\frac{\theta}{M}) $ for some $ i $, then $$ \lambda_i = \frac{\epsilon_2 \lambda_i}{1 + \epsilon_2} + \frac{\lambda_i}{1 + \epsilon_2} \leq \lambda_i \epsilon_2 + \frac{\lambda_i}{1 + \epsilon_2} \leq \epsilon_1 + \frac{\lambda_i}{1 + \epsilon_2} \leq \epsilon_1 + \gamma_i $$by the choice of $ M, \epsilon_2$. Hence, $$ \sum_{i \in I: 0 \leq \lambda_i \leq \cot(\frac{\theta}{M})} \text{arccot}(\gamma_i + \epsilon_1) \leq \sum_{i \in I: 0 \leq \lambda_i \leq \cot(\frac{\theta}{M})} \text{arccot}(\lambda_i) .$$Finally, if $ \lambda_i > \cot(\frac{\theta}{M}) $ for some $ i $, then $$ \epsilon_1 + \gamma_i \geq \epsilon_1 + \frac{\lambda_i}{1 + \epsilon_2} > \cot(\frac{\theta}{M}) $$again by the choice of $ M  $ and $ \epsilon_2 $ so that $$ \sum_{i \in I: \lambda_i > \cot(\frac{\theta}{M})} \text{arccot}(\gamma_i + \epsilon_1) \leq \frac{n \theta}{M} < \frac{\epsilon_3}{2} $$by the choice of $ M $. Hence, we get $$ \sum_{i \in I} \text{arccot}(\gamma_i + \epsilon_1) < \sum_{i \in I: \lambda_i \leq \cot(\frac{\theta}{M})} \text{arccot}(\lambda_i) + \frac{\epsilon_3}{2} < \theta - \frac{\epsilon_3}{2} $$and $ \omega + \epsilon_1 \chi \in \Gamma_{\chi, \theta} $ as required.
\end{proof}	
	
	The next lemma will be used mainly while proving the concentration of mass result.
\begin{lem}
	\label{lem:bound}
		\textbf{a}. If $ \omega \in \Gamma_{\chi, \theta} $, then there exists a $ C $ depending only on $ \theta, n $ such 
		that $ C( \Omega_{\theta} + \chi )^{n} = C(\omega - \cot(\theta) \chi + \chi)^{n} \geq G_\theta^{n}(\omega, \chi) $. \\
		\textbf{b}. If $ \omega \in \Gamma_{\chi, \theta - \delta} $ for some $ \delta \in (0, \theta) $, then letting $ \Omega = \omega - \chi \cot(\theta) $, we have $ P_\theta^{k}(\Omega, \chi) \geq \epsilon \cdot \Omega^{k} $ for some $ \epsilon > 0 $ depending only on $ \theta, n, \delta $ and $ k = 1, \dots, n - 1 $.	
	
\end{lem}
\begin{proof}
	a. If $ \lambda - \cot(\theta) > 0 $, then there exists a $ C > 0 $ independent of $ \lambda $ such that $ C(\lambda - \cot(\theta) + 1) \geq \sqrt{1 + \lambda^{2}} $ from this and the expression for $ G^n_{\theta} $, part $ a $ follows. \\
	b. There exists a $ C > 0 $ such that $ \sqrt{1 + \lambda^2} \geq C(\lambda - \cot(\theta)) $ for any $ \lambda $ and by the expression for $ G^k_{\theta} $ for $ k < n $, part $ b $ follows.
\end{proof}

	\section{The twisted dHYM equation and concentration of mass} 
	
	In this section, we will prove Theorem \ref{mainpdethm}. As before, let $ (M, \chi) $ be a K\"ahler manifold of dimension $ n \leq 3 $ and let $ \omega $ be a smooth, closed, real $ (1, 1) $-form on $ M $. Let $ f $ be a smooth function on $ M $ which satisfies $ \displaystyle \int_M \Re(\omega +  \sqrt{-1} \chi )^n - \cot(\theta) \Im(\omega +  \sqrt{-1} \chi)^{n} = \int_M f \chi^{n} \geq 0 $ for some $ \theta \in (0, \pi) $. Suppose also that $ \omega \in \Gamma_{\chi, \theta} $ with $ 0 < \theta  < \pi $. We will now solve the differential equation 
	$ \displaystyle \Re(\omega_\phi +  \sqrt{-1} \chi )^n - \cot(\theta) \Im(\omega_\phi +  \sqrt{-1} \chi)^{n} = f \chi^{n} $ for an $ \omega_\phi = \omega +  \sqrt{-1} \partial \bar{\partial} \phi $ $ \in \Gamma_{\chi, \theta} $ assuming some conditions on $ f $. In fact, we will first show that there exists an $ \epsilon > 0 $ depending only on $ \dim(M), \theta $ such that if $ f > -\epsilon $, the above equation can be solved for an $ \omega_\phi \in \Gamma_{\chi, \theta} $.
	\\
	
	If $ n = \dim(M) = 1 $, the theorem follows immediately from the fact that $ H^{(1,1)}(M, \mathbb{R}) = \mathbb{R} $: as $ \int_M \omega - \cot(\theta) \chi = \int_M f \chi $, there exists a smooth $ \phi $ such that $ \omega_\phi  - \cot(\theta) \chi = f \chi $. The condition that $ \omega_\phi \in \Gamma_{\chi, \theta} $ is vacuous in this case. \\
	
	If $ n = 2 $, then we have $ \Omega := \omega - \cot(\theta) \chi > 0 $ and we are reduced to solving $ \displaystyle \Omega_\phi^2  = (\csc^2(\theta) + f) \chi^2 $ for an $ \Omega_\phi = \Omega +  \sqrt{-1} \partial \bar{\partial} \phi > 0 $. If $ f > -\csc^{2}(\theta) $, then the right side of the equation is positive, so by \cite{yaucalab}, we can find the required $ \Omega_\phi $. \\
	
	We now come to the main case: $ n = 3 $. Let $ \Omega = \omega - \cot(\theta) \chi $ as before. The condition $ \omega \in \Gamma_{\chi, \theta} $ can be written as 
	\begin{equation}
	\begin{split}
	\label{suffice}
	\Omega &> 0 \\
	\Omega^2 - \csc^{2}(\theta) \chi ^{2} &> 0 
	\end{split} 
	\end{equation}
	
	For simplicity of notation, we will replace $ f $ by $ 2 \csc^2(\theta) f $ so that the differential equation we need to solve becomes 
	\begin{equation}
	\label{maineq}
	\Omega_\phi^3 = 3 \csc^2(\theta)  \chi^2 \Omega_\phi + 2 \csc^2(\theta) (f + \cot(\theta))  \chi ^3
	\end{equation}
	
	The result we will prove is 
	\begin{prop}
		\label{mainthm}
		Define $ \epsilon := \csc(\theta) - |\cot(\theta)| $. If $ f > -\epsilon  $ and $ \int_M f \chi ^{3} \geq 0  $, Equation \ref{maineq} has a solution $ \Omega_\phi \in  C_{\chi, \theta}  $.
	\end{prop}
	
	Consider the following family of PDE for smooth real functions $ \phi_t $, for $ t \in [0, 1] $:
	\begin{equation}
	\label{maincont}
	\Omega_{\phi_t}^3 = 3 \csc^2(\theta)  \chi^2 \Omega_{\phi_t} + 2 \csc^2(\theta) (tf + \cot(\theta) + d_t)  \chi ^3
	\end{equation} 
	where $ d_t $ is a constant chosen so that the integral of both sides (which depends only on [$ \Omega $], [$  \chi  $], $ t $, and $ \theta $) are equal. It can be seen that
	\begin{equation}
		d_t \int_M \chi^{3} = (1 - t) \int_M f \chi^{3} \geq 0.
	\end{equation}
	 We also write $ z_t = tf + d_t + \cot(\theta) $ from now on. Let $ S $ be the set of all $ t $ in $ [0, 1] $ such that Equation \ref{maincont} has a smooth solution $ \phi_t $ such that 
	\begin{equation}
	\label{subsol}
	\begin{split}
	\Omega_{\phi_t} > 0 \\
	\Omega_{\phi_t}^2 - \csc^{2}(\theta)\chi^2 > 0
	\end{split}	
	\end{equation}
	We will show that $ S = [0, 1] $ by showing that $ S $ is non-empty, open, and closed in $ [0, 1] $. \\
	
	Let $ \phi $ be a solution of Equation \eqref{maineq} at some $ t $ and suppose that a point $ p $ on $ M $, holomorphic normal coordinates for $ \chi $ have been chosen so that $ \Omega_{\phi}(p)  = \sqrt{-1} \sum_i \lambda_{i} dz^{i} \wedge d\bar{z}^{i} $ in these coordinates. Assume that $ \lambda_1 \geq \lambda_2 \geq \lambda_3 $. At $ p $, equation \ref{maincont} and conditions \ref{subsol} can be written as 
	\begin{equation}
	\label{eigenform}
	\begin{split}
	\lambda_{1} \lambda_{2} \lambda_{3} &= \csc^{2}(\theta)(\lambda_1 + \lambda_2 + \lambda_3 + 2z_t) \\
	\lambda_i &> 0, \quad \forall i \\
	\lambda_i \lambda_j &>  \csc^{2}(\theta), \quad \forall i \neq j
	\end{split}
	\end{equation}
	\\
	
	As $ \lambda_1^{2} \geq \lambda_{1} \lambda_{2} >  \csc^{2}(\theta)   $, we see that $ \lambda_1 $ is bounded below independent of $ t $ and similarly for $ \lambda_{2} $. \\
	
	Since $ M $ is compact and $ f > -\epsilon $, we have $ f > -\epsilon + \delta $ for some $ \delta > 0 $ which we can assume to be less than $  \epsilon $ so 
	\begin{equation}
	\label{sumbound}
	\begin{split}
	\lambda_i + \lambda_j + 2z_t &\geq 2(\sqrt{\lambda_{i} \lambda_{j}} + z_t) \\
	&> 2(\csc(\theta) + z_t) \\
	&\geq 2(\csc(\theta) + \cot(\theta) + tf + d_t ) \\
	& > 2(\epsilon + t(\delta - \epsilon)) \\
	& > 2 \delta.
	\end{split}
	\end{equation}
	
	We now proceed with the proof of Theorem \ref{mainthm}.
	\begin{prop}
		\label{open}
		$ S $ is non-empty and open.
	\end{prop}
	\begin{proof}
		At $ t = 0 $, Equation \eqref{maincont} reduces to $ \Omega_{\phi}^3 - 3 \csc^2(\theta)  \chi^2 \Omega_{\phi} - 2 \csc^{2}(\theta) \cot(\theta) \chi^{3} = 2 \csc^2(\theta) d_0 \chi^{3}  $. As $ d_0 \geq 0 $ is a constant, $ \Omega > 0$, $ \Omega^2 - \csc^2(\theta)  \chi ^2 > 0 $, by \cite{gao} the equation has a solution $ \phi $ at $ t = 0 $ with $ \Omega_{\phi} > 0  $ and $ \Omega_{\phi}^{2} - \csc^{2}(\theta) \chi^2 > 0 $. \\
		
		The linearization of Equation \ref{maincont} at time $ t \in S $ is $ Lu = (3\Omega_{\phi_t}^2 - 3  \csc^2(\theta)  \chi ^2) \sqrt{-1} \partial \bar{\partial} u $. By the definition of $ S $, this linear operator is strongly elliptic and hence there exist solutions $ \phi_s $ for all $ s $ close to $ t $ which depend smoothly on $ s $, so the conditions continue to hold for these solutions. Thus, $ S $ is open.  
	\end{proof}
	
	To prove that $ S $ is closed, we will need to find various bounds on the derivatives of the solutions as is usual in the method of continuity. 
	\begin{prop}[$C_0$ bounds]
		\label{C0bound}
		If $ \phi_t $ is a solution of Equation \eqref{maincont} for some $ t $ with $ \sup \phi_t = 0 $, then $ \inf \phi_t \geq -C $, where the constant $ C > 0 $ is independent $ t $ and $ \phi_t $.
	\end{prop}
	\begin{proof}
		The proof of this Proposition is the same as that of Lemma 3.1 and Proposition 3.1 of \cite{pin3}, based on \cite{blocki1}, \cite{blocki2}. 
	\end{proof}
	
	For a positive definite matrix $ A $ and a real parameter $ h $, we define 
	\begin{equation}
	\label{Fdef}
	F(h, A) = \frac{ \text{tr}(A) + 2h}{\text{det}(A)}
	\end{equation}
	If $ A = \chi^{-1} \Omega_{\phi_t} $, Equation \eqref{maincont} can then be written as 
	\begin{equation}
	\label{Feq}
	F(z_t(x), A(x)) = \sin^{2}(\theta) 
	\end{equation}  
	\begin{prop}[Laplacian estimates]
		\label{Lapbound}
		For a solution $ \phi_t $ of \ref{maincont}, we have $$ || \Delta \phi_t ||_{C_0} \leq C(1 + ||\nabla \phi_t||_{C_0}^2) $$ where $ C $ is independent of $ t $.
	\end{prop}
	\begin{proof}
		This proposition too follows the same method as \cite{pin3}, but the presence of some extra third order derivatives of $ \phi $ due to $ f $ not necessarily being constant necessitates some more delicate estimates.
		\\
		
		Let $ \lambda_{1} $ be the largest eigenvalue of $  \chi ^{-1} \Omega_{\phi_t} $ and let $$ \psi(x, \phi) = -\gamma(\phi(x)) + \log(\lambda_{1}(x)) $$ where $$ \gamma(u)=2 Au-\frac{A \tau}{2}u^{2} $$with $ 0 \leq u \leq C_0 $
		(We normalize the $ \phi $ by adding a constant so that its range lies in $ [0, C_0] $ for some uniform $ C_0 > 0 $). The constants $ A $ and $ \tau $ will be chosen as in \cite{pin3} \\
		
		If $ \phi $ is a solution to Equation \eqref{maincont} at some $ t $, then $ \Delta \phi > -C $ for some $ C $ independent of $ t $ as $ \Omega_{\phi} > 0 $. Hence, if the proposition is false, there exists a sequence $ t_n $ and solutions $ \phi_n $ to Equation \eqref{maincont} such that $ \max \Delta \phi_n > n(1 + ||\nabla \phi_n||_{C_0}^2) $. Let $ \psi_n(x) = \psi(x, \phi_n) $. Let $ p_n $ be a point where $ \psi_n $ attains its maximum. As $ \gamma $ is bounded independent of $ n $ and $ 3 \max \lambda_{1, n} \geq \max \Delta \phi_n $ is unbounded, we see that $ \phi_n(p_n) $ and hence $\lambda_{1, n}(p_n) $ is unbounded. We suppress the subscript $ n $ from here on.\\
		
		We now note that because $$ \lambda_{3}  = \frac{\csc^{2}(\theta)(\lambda_{1} + \lambda_2 + 2z_t)}{\lambda_{1} \lambda_2 -  \csc^{2}(\theta)}$$ and as $ \lambda_2 $ is bounded below by positive constant, we have  
		\begin{equation}
		\label{prodlim}
		\lambda_3 \lambda_2 -  \csc^{2}(\theta) \rightarrow 0 
		\end{equation}
		when $ \lambda_1 \rightarrow \infty $. \\

		Let $ p \in M $ and assume that holomorphic normal coordinates have been chosen for $ \chi $ at $ p $ such that $ \Omega_{\phi_t}(p) = \sum_{i} \sqrt{-1} \lambda_i dz^i \wedge d\bar{z}^i $ is diagonal at $ p $. Also assume $ \lambda_1 \geq \lambda_2 \geq \lambda_3 $. In what follows, we will write $ \phi $ instead of $ \phi_t $. If $ p $ is the point where $ \psi $ attains its maximum, we wish to differentiate $ \psi $ at $ p $ and apply the maximum principle but $ \lambda_{1} $ need not be smooth at $ p $, so the standard remedy here is to first perturb $ \Omega_{\phi} $ in a neighborhood of $ p $. Define (on a neighborhood of $ p $) $ \displaystyle \tilde{\Omega} := \Omega_{\phi} -  \sqrt{-1} B_2 dz^2 \wedge d \bar{z}^2 -  \sqrt{-1}  B_3 dz^3 \wedge d\bar{z}^3 $ where $ B_2, B_3 $ are small constants so that $ \displaystyle \tilde{\Omega} $ is positive definite in a neighborhood of $ p $ and such that $ \chi ^{-1} \tilde{\Omega} $ has eigenvalues $ \tilde{\lambda}_1 \geq \tilde{\lambda}_2 \geq \tilde{\lambda}_3  $. At $ p $, $ \lambda_1 = \tilde{\lambda}_1 $. The function $ \displaystyle \tilde{\psi} $ obtained by replacing $ \lambda_{1} $ with $ \tilde{\lambda}_1 $ in the definition of $ \psi $ also attains a maximum at $ p $. Differentiating $ \displaystyle \tilde{\psi} $ at $ p $, we get $$ 0 = -\gamma' \phi_{,\mu} + \frac{1}{\lambda_1} (\Omega_{\phi})_{11, \mu} $$ and $$ 0 \geq -\frac{\partial F}{\partial \lambda_\mu} \tilde{\psi}_{,\mu \bar{\mu}} $$
		
		As in \cite{pin3}, this last inequality can be written as $ 0 \geq \alpha + \beta $, where 
		\begin{equation}
		\label{decompa}
		\begin{split}
		\alpha & =  \sum_{\mu}\left(-\frac{\partial F}{\partial \lambda_{\mu}}\right)\left(\frac{1}{\lambda_{1}} \tilde{\lambda}_{1, \mu}\right)_{, \bar{\mu}} \\
		& = \sum_{\mu} \left( -\frac{\partial F}{\partial \lambda_{\mu}} \right) \left(  -\frac{1}{ \lambda_{1}^{2} }|(\Omega_{\phi})_{1\bar{1}, \mu}|^{2} \right.  + \\ &  \left. \frac{1}{\lambda_{1}} \left[ R^{1 \bar{1}}_{\mu \bar{\mu}} + (\Omega_{\phi})_{\mu \bar{\mu}, 1 \bar{1}} + \sum_{q > 1}\frac{|(\Omega_{\phi})_{q\bar{1}, \mu}|^{2} + |(\Omega_{\phi})_{1\bar{q}, \mu}|^{2}}{\lambda_{1} - \tilde{\lambda_{q}}} \right] \right) 
		\end{split}		
		\end{equation}
		and $$ \beta = \sum_{\mu} \frac{\partial F}{\partial \lambda_{\mu}} (\gamma' \phi_{,\mu})_{,\bar{\mu}} $$
		
		The $ \beta $ term is dealt with exactly the same way as the $ B $ term in \cite{pin3}. For $ \alpha $, since instead of Equations 3.23 and 3.24 in \cite{pin3}, we have \eqref{linconst} and \eqref{doublediff} below, we use a modification of Lemma 3.2 in \cite{pin3}. \\

		We will now differentiate the equation $ F(z(x), A(x)) = \sin^2(\theta) $ twice at the point $ x = p $, so we make the following computations in advance. When the matrix $ A $ is diagonal, i.e. $ A_{\mu {\nu}} = \lambda_\mu \delta_{\mu \nu } $, we have: \\
		
		\begin{equation*}
		 \frac{ \partial F}{\partial A_{\mu {\nu}} } = \frac{ \delta_{\mu \nu } - ( \text{Tr}(A) + 2 z_t)(A^{-1})_{\nu {\mu} }}{\text{det}(A)} 
		\end{equation*}
		 		
		 \begin{equation*}
		 \frac{ \partial^2 F}{ \partial A_{ \alpha {\beta} } \partial A_{\mu {\nu} } } = \frac{1}{\text{det}(A)} \bigg(-\frac{ \delta_{\mu \nu } \delta_{\alpha \beta }}{\lambda_\alpha} -\frac{ \delta_{\mu \nu } \delta_{\alpha \beta }}{\lambda_\mu} + ( \text{Tr}(A) + 2z_t )\bigg(\frac{\delta_{\mu \nu } \delta_{\alpha \beta } + \delta_{\alpha \nu } \delta_{\mu \beta } }{ \lambda_\mu \lambda_\alpha } \bigg) \bigg) 	
		 \end{equation*}
		
		 \begin{equation*}
		 \frac{\partial F}{\partial h} = \frac{2}{\text{det}(A)}
		 \end{equation*} 
			
		\begin{equation*}
		\frac{\partial^{2} F}{\partial h \partial A_{\mu {\nu}}} = \frac{2}{\text{det}(A)}(A^{-1})_{\nu {\mu}}
		\end{equation*} 	
		
		When $ A(x)_{\mu \nu} = (\Omega_\phi)_{\mu \bar{\beta}}(x) \omega^{\bar{\beta} \nu}(x) $, we have $ F(z(x), A(x)) = \sin^{2}(\theta) $ for all $ x $. We now differentiate the function $ x \mapsto F(z(x), A(x)) $ at the point $ p \in M $ with respect to the coordinate $ z^1 $ to get:
		
		\begin{equation}
		\label{z1deriv}
			\frac{2t}{\text{det}(A)} f_{, 1} + \sum \frac{\partial F}{\partial A_{\mu \mu}} (\Omega_\phi)_{\mu \bar{\mu}, 1} = 0
		\end{equation}
		where $ {}_{, 1} $ denotes the partial derivative with respect to $ z^1 $. If $ z^1 = x +  \sqrt{-1} y $, then we have $ \partial_1 = \displaystyle \frac{\partial_x - \sqrt{-1} \partial_y}{2} $ and $ \displaystyle \partial_{\bar{1}} = \frac{\partial_x + \sqrt{-1} \partial_y}{2} $. Equating real and imaginary parts in the equation \ref{z1deriv} gives 
		\begin{equation}
		\label{linconst}
		\sum_{\mu} ( \sum_{i \neq \mu} \lambda_{i} + 2z_t)\frac{(\Omega_{\phi})_{\mu \bar{\mu}, j}}{\lambda_{\mu}} = 2tf_j
		\end{equation}
		for $ j = x, y $. Defining $ \displaystyle v_\mu := \frac{(\Omega_{\phi})_{\mu \bar{\mu}, x}}{\lambda_{\mu}} $ and $ \displaystyle w_\mu := \frac{(\Omega_{\phi})_{\mu \bar{\mu}, y}}{\lambda_{\mu}} $, equations \ref{linconst} can be written as:
		
		\begin{equation}
		\label{xeliminate}
			\sum_{ \mu }v_\mu ( \sum_{k \neq \mu } \lambda_k + 2z_t) = 2t f_{, x}
		\end{equation}
		\begin{equation}
		\label{yeliminate}
		\sum_{ \mu }w_\mu ( \sum_{k \neq \mu } \lambda_k + 2z_t) = 2t f_{, y}
		\end{equation}
		
		Now we differentiate $ x \mapsto F(z(x), A(x)) $ twice: first with respect to $ z^1 $ and then with respect to $ \bar{z}^1 $ to get (recall that $ \omega_{, 1}(p) = 0 $ as we have chosen holomorphic normal coordinates at $ p $)
		
		\begin{equation}
		\label{doublediff}
		\begin{split}
		-\sum_{\mu} \frac{\partial F}{ \partial \lambda_{\mu}} (\Omega_{\phi})_{\mu \bar{\mu}, 1 \bar{1}} = &\sum_\mu \frac{\partial F}{ \partial \lambda_{\mu}} \lambda_{\mu} R^{\mu \bar{\mu}}_{1 \bar{1}} + \color{blue} \sum_{\mu, \nu, \alpha, \beta} \frac{\partial^2 F}{\partial A_{\mu \bar{\nu}} \partial A_{\alpha \bar{\beta}}} (\Omega_{\phi})_{\mu \bar{\nu}, 1} (\Omega_{\phi})_{\alpha \bar{\beta}, \bar{1}}  \\& + \color{blue} 
		\sum_{\mu} \left( -\frac{2t}{\lambda_{\mu} \det}(f_1 (\Omega_{\phi})_{\mu \bar{\mu}, \bar{1}} + f_{\bar{1}} (\Omega_{\phi})_{\mu \bar{\mu}, 1} ) \right)  + \color{black} \frac{2t}{\det} f_{1\bar{1}} 
		\end{split}	
		\end{equation}	
			
We now simplify the $ \color{blue} \text{middle two terms}$ in Equation \eqref{doublediff}. As 
\begin{equation*}
	\Omega_{\mu \bar{\nu}, 1} \Omega_{\alpha \bar{\beta}, \bar{1}} =	\frac{1}{4} (\Omega_{\mu \bar{\nu}, x} \Omega_{\alpha \bar{\beta}, x} + \Omega_{\mu \bar{\nu}, y} \Omega_{\alpha \bar{\beta}, y}  +  \sqrt{-1} (\Omega_{\mu \bar{\nu}, x} \Omega_{\alpha \bar{\beta}, y} - \Omega_{\mu \bar{\nu}, y} \Omega_{\alpha \bar{\beta}, x}) ), 
\end{equation*}we see 
\begin{equation*}
	\sum_{\mu, \nu, \alpha, \beta} \frac{\partial^2 F}{\partial A_{\mu {\nu}} \partial A_{\alpha {\beta}}} (\Omega_{\phi})_{\mu \bar{\nu}, 1} (\Omega_{\phi})_{\alpha \bar{\beta}, \bar{1}} = \frac{1}{4} \sum_{\mu, \nu, \alpha, \beta} \frac{\partial^2 F}{\partial A_{\mu {\nu}} \partial A_{\alpha {\beta}}} ( \Omega_{\mu \bar{\nu}, x} \Omega_{\alpha \bar{\beta}, x} + \Omega_{\mu \bar{\nu}, y} \Omega_{\alpha \bar{\beta}, y} ).
\end{equation*}
		
By the expression for $ \displaystyle \frac{\partial^2 F}{\partial A_{\mu {\nu}} \partial A_{\alpha {\beta}}}  $ computed earlier, 
\begin{equation*}
\begin{split}
	\frac{1}{4} \sum_{\mu, \nu, \alpha, \beta} \frac{\partial^2 F}{\partial A_{\mu {\nu}} \partial A_{\alpha {\beta}}}  \Omega_{\mu \bar{\nu}, x} \Omega_{\alpha \bar{\beta}, x} = \\
	\frac{1}{4 \cdot \text{det}(A)} \sum_{ \mu, \alpha } \frac{\Omega_{\mu \bar{\mu}, x} \Omega_{\alpha \bar{\alpha}, x} (\text{tr(A)} - \lambda_\mu - \lambda_\alpha + 2z_t ) }{\lambda_{\mu} \lambda_{\alpha} } +  \\
	\frac{1}{4 \cdot\text{det}(A)} \sum_{ \mu, \alpha } \frac{|\Omega_{\mu \bar{\alpha}, x}|^{2} ( \text{tr(A)} + 2z_t ) }{\lambda_{\mu} \lambda_{\alpha} }	
\end{split}
\end{equation*}
In terms of the $ v_\mu $'s this is equal to 
\begin{equation}
\label{doublex}
\begin{split}
	\frac{1}{4} \sum_{\mu, \nu, \alpha, \beta} \frac{\partial^2 F}{\partial A_{\mu {\nu}} \partial A_{\alpha {\beta}}}  \Omega_{\mu \bar{\nu}, x} \Omega_{\alpha \bar{\beta}, x} = & \\
	\frac{1}{4 \cdot \text{det}(A)} \sum_{ \mu, \alpha } {v_\mu v_\alpha ((\text{tr(A)} - \lambda_\mu - \lambda_\alpha) + 2z_t ) } + &
	\frac{1}{4 \cdot \text{det}(A)} \sum_{ \mu } v_\mu^2( \text{tr}(A) + 2z_t) + \\
	\frac{1}{4 \cdot \text{det}(A)} \sum_{ \mu \neq \alpha } \frac{|\Omega_{\mu \bar{\alpha}, x}|^{2} ( \text{tr(A)} + 2z_t ) }{\lambda_{\mu} \lambda_{\alpha} } &
	\end{split}
\end{equation}
Similarly, we have 
\begin{equation}
\label{doubley}
	\begin{split}
	 \frac{1}{4} \sum_{\mu, \nu, \alpha, \beta} \frac{\partial^2 F}{\partial A_{\mu {\nu}} \partial A_{\alpha {\beta}}}  \Omega_{\mu \bar{\nu}, y} \Omega_{\alpha \bar{\beta}, y}  = &  \\
	 \frac{1}{4 \cdot \text{det}(A)} \sum_{ \mu, \alpha } {w_\mu w_\alpha (\text{tr(A)} - \lambda_\mu - \lambda_\alpha + 2z_t ) } & + 
	 \frac{1}{4 \cdot \text{det}(A)} \sum_{ \mu } w_\mu^2( \text{tr}(A) + 2z_t) +  \\
	  \frac{1}{4 \cdot \text{det}(A)} \sum_{ \mu \neq \alpha } \frac{|\Omega_{\mu \bar{\alpha}, y}|^{2} ( \text{tr(A)} + 2z_t ) }{\lambda_{\mu} \lambda_{\alpha} } &
	\end{split}
\end{equation} 
The other highlighted summation in \ref{doublediff} can be written as
\begin{equation}
	\label{derivf}
	\sum_{\mu} \left( -\frac{2t}{\lambda_{\mu} \det(A)}(f_1 (\Omega_{\phi})_{\mu \bar{\mu}, \bar{1}} + f_{\bar{1}} (\Omega_{\phi})_{\mu \bar{\mu}, 1} ) \right) = - \frac{t}{\det(A)} \sum_{\mu} \left( f_{, x} v_\mu + f_{,y} w_\mu \right)
\end{equation}
Now, using equations \ref{xeliminate} and \ref{yeliminate}, $ v_3 $, $ w_3 $ can be eliminated from equations \ref{doublex}, \ref{doubley}, \ref{derivf} and the resulting equations can be added to give the sum of the highlighted terms in \ref{doublediff} as

		\begin{equation}
		\label{calcresult}
		\begin{split}
		&\frac{ \sum_{i} \lambda_{i} + 2z_t }{\lambda_1 \lambda_2 \lambda_3 (\lambda_1 + \lambda_2 + 2z_t)}\sum_{\mu \neq \alpha} \frac{|(\Omega_{\phi})_{\mu \bar{\alpha}, 1}|^{2}}{\lambda_{\mu}\lambda_{\alpha}}  \\
		&+ \frac{av_1^2 + 2bv_1v_2 + cv_2^2 + dv_1 + ev_2 }{4 \lambda_1 \lambda_2 \lambda_3 (\lambda_1 + \lambda_2 + 2z_t)} \\
		&+ \frac{aw_1^2 + 2bw_1w_2 + cw_2^2 + d'w_1 + e'w_2 }{4 \lambda_1 \lambda_2 \lambda_3 (\lambda_1 + \lambda_2 + 2z_t)}
		\end{split}	
		\end{equation}
		where
		\begin{equation*}
		\label{coeff}
		\begin{split}
		a & = 2(\lambda_2 + \lambda_{3} + 2z_t)( \sum_i \lambda_i + 2z_t) \\
		b & = 2(\lambda_3 + z_t )(\sum_{i} \lambda_i + 2z_t)\\
		c & = 2(\lambda_1 + \lambda_{3} + 2z_t)( \sum_i \lambda_i + 2z_t)\\
		d & = -4tf_x(\lambda_1 + \lambda_2 + \lambda_3 + 2z_t) \\
		e & = -4tf_x(\lambda_1 + \lambda_2 + \lambda_3 + 2z_t)\\
		\end{split}	
		\end{equation*}
		and the corresponding expressions for the quantities $ d', e'$ are obtained by replacing $ x $ by $ y $ in $ d, e $. From now on, we will only deal with the unprimed quadratic as the primed one can be dealt with the same techniques. \\
		
		To estimate the quadratic $av_1^2 + 2bv_1v_2 + cv_2^2 + dv_1 + ev_2$  in Equation \eqref{calcresult}, we first show that these quadratics are convex and bounded below. As $a, c > 0  $, and as the discriminant of the quadratic is $$ \Delta = ( \sum_i \lambda_i + 2z_t)^2( \sum_{i < j} \lambda_{i} \lambda_j + 2z_t \sum_{i} \lambda_i + 3z_t^2), $$it suffices to show that $$g = \sum_{i < j} \lambda_{i} \lambda_j + 2z_t \sum_{i} \lambda_i + 3z_t^2 > 0$$. 
		
		\begin{lem}
			\label{discpos}
			If $ \lambda_i $ satisfy the equation $$ \lambda_{1} \lambda_{2} \lambda_{3} = \csc^{2}(\theta)(\lambda_{1} + \lambda_{2} + \lambda_{3} + 2z_t) $$ and the inequalities $ \lambda_i > 0 $, $ \lambda_{i} \lambda_{j} >  \csc^{2}(\theta)  $, then $ g > 0 $.
		\end{lem}
		\begin{proof}
			
			This is trivial if $ z_t \geq 0 $, so we only consider the $ z < 0 $ case (Hence, we must have $ 0 > z_t  > -\csc(\theta) $). Writing 
			\begin{equation*}
			\begin{split}
			g = ( \lambda_1 + 2z_t)(\lambda_3 + \lambda_2 + 2z_t) + \lambda_2 \lambda_3 - z_t^2 \\ \geq ( \lambda_1 + 2z_t)(\lambda_3 + \lambda_2 + 2z_t) + \csc^2(\theta) - z_t^2 
			\end{split}	
			\end{equation*} we see that $ g > 0 $ on points of $ \bar{C} $ where $  \lambda_{i} + 2z_t \geq 0 $ for any $ i $ because $ \csc^2(\theta) - z_t^2 > 0$. In particular, this holds if any $ \lambda_{i} $ is sufficiently large or small (if, say, $ \lambda_3 $ is small, then $ \lambda_{2} $ is large because $ \lambda_{2} \lambda_{3} >  \csc^{2}(\theta)  $). Also, if, say, $ \lambda_{2} \lambda_{3} -  \csc^{2}(\theta) $ is small, then $ \lambda_{1} $ is large and hence $ g > 0$. So as the $ \lambda_i $s approach the boundary of the constraints, $ g > 0 $.  
			\\
			
			Using Lagrange multipliers, the critical points of $ g $ subject to the constraint $  \sum_i \lambda_i + 2z_t = \sin^2(\theta) \lambda_1 \lambda_2 \lambda_3 $ are seen to occur when $ \lambda_i = \lambda $ for all $ i $ and some constant $ \lambda $. At such points, it is easy to see that $ g > 0 $. Hence, $ g > 0 $ everywhere (on the constrained set).
			
		\end{proof}

		Thus, $av_1^2 + 2bv_1v_2 + cv_2^2 + dv_1 + ev_2$ is bounded below and has a unique critical point. The value of this quadratic at this critical point is seen to be
		\begin{equation}
		\label{minval}
		\begin{split}
		\min = \frac{\text{num}}{\text{den}} = & \frac{ac^2d^2 -ab^2e^2 +a^2ce^2 -b^2cd^2 + 2b^3de - 2abcde}{4\Delta^2} \\ & + \frac{2 \Delta (2bde -cd^2 -ae^2)}{4 \Delta^{2}} 
		\end{split}
		\end{equation}
		
		From the expressions for the coefficients of the quadratic and that of the minimum value, it can be seen that the numerator is a polynomial in the $ \lambda_{i} $s of degree at most $ 8 $ with coefficients which can be bounded uniformly (independent of the point $p$ and $t$). Hence, for large enough $ \lambda_1 $
		\begin{equation}
		\label{numbound}
		\text{num} \geq -C \lambda_{1}^8
		\end{equation}for a uniform constant $ C $. \\
		
		To obtain a positive lower bound for the denominator $ \text{den} = 4 \Delta^2 $, we notice that the coefficient of $ \lambda_{1} $ in $ g $ is $ \displaystyle \lambda_{2} + \lambda_3 + 2z_t $, which is bounded below by a positive constant depending only on $ \epsilon $, $ \delta $ by \ref{sumbound}. \\		
		
		Therefore, for large enough $ \lambda_{1} $, $$ \Delta = 4(\sum_i \lambda_{i} + 2z_t)^2 g \geq C\lambda_1^2 \lambda_1 = C\lambda_{1}^{3} $$ for some uniform constant $ C > 0$. Also, we observe that $ \lambda_1 \lambda_2 \lambda_3 > \lambda_{1}  \csc^{2}(\theta) $ and $ (\lambda_1 + \lambda_2) + 2z_t \geq C\lambda_{1} $ for large enough $ \lambda_{1} $. Combining all these estimates, we see that for large $ \lambda_1 $, $  \lambda_1 \lambda_2 \lambda_3  ((\lambda_1 + \lambda_2) + 2z_t) \text{den} > C\lambda_1^{8}$ and hence,
		
		\begin{equation*}
		\label{resultbound}
		\eqref{calcresult} \geq \frac{ \sum_{i} \lambda_{i} + 2z_t }{\lambda_1 \lambda_2 \lambda_3 ((\lambda_1 + \lambda_2) + 2z_t)}\sum_{\mu \neq \alpha} \frac{|(\Omega_{\phi})_{\mu \bar{\alpha}, 1}|^{2}}{\lambda_{\mu}\lambda_{\alpha}} - C
		\end{equation*}
		for a uniform constant $ C $. \\
		
		From this point onwards, we proceed exactly as in \cite{pin3} and obtain $$ \lambda_{1} < C(1 + |\nabla \phi|^2) $$ for some $ C $ independent of $ t $ and the point $ p $ on $ M $. Thus, we must have $ |\Delta \phi| \leq C(1 + |\nabla \phi|^2) $ for some uniform constant $ C $.
		
	\end{proof}
	The blow-up argument in \cite{cjy} now shows that $ ||\Delta \phi|| < C $ for some uniform constant and we consequently obtain bounds on $ \sqrt{-1} \partial \bar{\partial} \phi $. We also see that the operator $ F $ is uniformly elliptic.
	\\
	
	\begin{prop}[$C^{2, \alpha}$ estimate]
		\label{holder}
		For a solution $ \phi $ of Equation \ref{maincont} at some $ t $, $||\phi||_{C^{2, \alpha}} \leq C $ for some $ C $ independent of $ \phi $ and $ t $.
	\end{prop}
	\begin{proof}
		If not, then we use a blow-up procedure as in \cite{cjy} to produce a point $ x_0 \in M $ and function $ u \in C^{3, \alpha}(\mathbb{C}^3) $ such that $ F(z(x_0), A) = \sin^{2}(\theta) $ and $ | \partial \bar{\partial} \partial u (0)| = 1 $, where $ A = \chi(x_0)^{-1} \partial \bar{\partial} u $, with $ \chi(x_0) $ treated as a Hermitian matrix defined using some coordinates at $ x_0 $. However, by Lemma 4.1 of \cite{pin3}, $ u $ must be a quadratic polynomial, contradicting $ |\partial \bar{\partial}\partial u(0)| = 1 $ (the convexity property used in the proof of the aforementioned Lemma still holds because when $ f $ is constant, the number $ \min $ is $ 0 $ and we get Lemma 3.2 of \cite{pin3}).
	\end{proof}
	Elliptic regularity and bootstrapping arguments along with the Arzela-Ascoli theorem now show that set of $ t \in [0, 1] $ where Equation \eqref{maincont} has a smooth solution is closed. The proof of Lemma 2.2 in \cite{pin3} shows that the conditions \ref{subsol} continue to hold at these points and hence $ S $ is closed, completing the proof of Theorem \ref{mainthm}.
	\\
	
	We now show how, if $ \omega $ solves 
	\begin{equation}
		\Re(\omega +  \sqrt{-1} \chi)^{n} - \cot(\theta)\Im(\omega +  \sqrt{-1} \chi)^{n} = f \chi^{n}
	\end{equation}
	with $ \omega \in \Gamma_{\chi, \theta} $, we can ensure that $ \omega \in \Gamma_{\chi, \theta, \Theta} $ by requiring $ f > -\epsilon $ for $ \epsilon > 0$ small enough. \\
	
	If $ n = 1 $, then we can simply take $ \epsilon = \cot(\theta) - \cot(\Theta) $, so we assume $ n > 1 $. Writing $ \lambda_1 \leq \dots \leq \lambda_n $ for the eigenvalues of $ \omega $ with respect to $ \chi $, the equation
	\begin{equation}
	\label{sinedhym}
		\prod_{i = 1}^{n} \sqrt{1 + \lambda_i^2} \cdot \frac{\sin(\theta - \sum_{ i = 1}^{n} \theta_i)}{\sin(\theta)} = f
	\end{equation}
	holds at each point of $ M $. If, at a point $ p $, $ f(p) \geq 0 $, then $ \displaystyle \sin(\theta - \sum_i \theta_i) \geq 0$ as well. As $ \omega \in \Gamma_{\chi, \theta} $, we can reason as Lemma \ref{sinelem} that this implies $ \displaystyle \theta \geq \sum_{i = 1}^n \theta_i $ at $ p $ and hence $ \omega \in \Gamma_{\chi, \theta, \theta} \subset \Gamma_{\chi, \theta, \Theta} $ at $ p $. \\
	
	On the other hand, if $ -\epsilon < f(p) < 0 $ for $ \displaystyle \epsilon := \frac{\sin(\frac{\Theta - \theta}{2})}{\sin(\theta)}$ , then we have
	\begin{equation}
	\label{sineineq}
		\sin(\theta - \sum_{ i  = 1}^{n} \theta_i) > -\epsilon \sin(\theta) = \sin(\frac{\theta - \Theta}{2})
	\end{equation}
	As $ \omega \in \Gamma_{\chi, \theta} $, we get $ \displaystyle \sin(\theta - \sum_{ i \in I } \theta_i) > 0 $ for all subsets $ I $ of $ \{1, \dots, n \} $ with $ |I| < n $. So by Lemma \ref{sinelem}, we get $ \displaystyle \sum_{i \in I} \theta_i < \theta $ for all such $ I $. In particular $ \displaystyle \sum_{i = 1 }^{n - 1} \theta_i < \theta  $. \\
	
	If $ n > 2 $, then $ \displaystyle \theta_n < \frac{\pi}{2} $ (if not, then $ \theta_1 + \theta_n \geq \pi > \theta $ contradicting $ \omega \in \Gamma_{\chi, \theta} $), so we have
	$$ -\frac{\pi}{2} < -\theta_n < \theta - \sum_{i = 1 }^{n} \theta_i < 0 $$at $ p $ (the leftmost inequality is due to the fact that the left-hand side of \ref{sinedhym} is negative at $ p $). We also have $ \displaystyle \frac{ \theta - \Theta}{2} \in (-\frac{\pi}{2}, 0) $ so by the monotonicity of $ \sin $ on $ (-\frac{\pi}{2}, 0) $ and the inequality \ref{sineineq}, $ \displaystyle \theta - \sum_{i = 1}^n \theta_i > \frac{ \theta - \Theta}{2} $ and hence $ \displaystyle \sum_{i = 1}^n \theta_i < \frac{ \theta + \Theta}{2} < \Theta $ .
	\\
	
	If $ n = 2 $ and $ \displaystyle \theta \leq \frac{\pi}{2} $, the above argument works without change as we still have $ \displaystyle -\frac{\pi}{2} < -\theta_2 < \theta - \theta_1 - \theta_2 < 0 $. If $ \displaystyle \theta > \frac{\pi}{2}  $ then it is not hard to see that $ \epsilon < 1 $, so $ \sin(\theta - \theta_1 - \theta_2 ) > -\sin(\theta) $. But we also have $ -\theta < \theta - \theta_1 - \theta_2 < 0 $ as $ \theta_1, \theta_2 < \theta $. If $ \displaystyle \theta - \theta_1 - \theta_2 \in [-\theta, \frac{\pi}{2} - \theta] $, then $ \displaystyle \sin(\theta - \theta_1 - \theta_2 ) \leq -\sin(\theta) $, so we must have $ \displaystyle \theta - \theta_1 - \theta_2 \in (\frac{\pi}{2} - \theta, 0 ) $. In particular, $ \displaystyle \theta - \theta_1 - \theta_2 \in (-\frac{\pi}{2}, 0 ) $ so the argument in the previous paragraph applies in this case too. Thus, $ \omega \in \Gamma_{\chi, \theta, \Theta} $. This completes the proof of Theorem \ref{mainpdethm}. \\

	Having proved the PDE result, we now sketch a proof of the following concentration of mass result analogous to \cite[Theorem~3.1]{datpin20}.  
	
	\begin{thm}
		\label{concofmass}
		Let $ (M, \chi) $ be a compact, connected K\"ahler manifold. Let $ \eta  $ be another K\"ahler form and suppose $ \Omega_0 $ is a smooth, real $ (1, 1) $-form such that for each $ t > 0 $, $ [\Omega_0 + t \eta] \in C_{\chi, \theta} $ (i.e. there exists a form in the class $ \Omega_0 + t \eta $ satisfying the cone condition). Assume that $ \displaystyle \int_M P^n(\Omega_0, \chi) \geq 0 $. Let $ Y $ be a codimension-1 subvariety of $ M $. Then there exists a current $ \Theta \geq \beta [Y]  $ for some $ \beta > 0 $ which satisfies the cone condition in the sense of definition \ref{def:conecurrent} on $ M $.
	\end{thm}
	\begin{proof}
		As in \cite[Lemma~2.1]{dempaun}, cover $ M $ by coordinate balls $ B_j $ such that on $ B_j $, $ Y $ is given by the zero set $ f_j $ of a holomorphic function on $ B_j $ and let $ \theta_j $ be a partition of unity subordinate to $ B_j $ such that $ \sum_j \theta_j^{2} = 1 $ on $ M $. Define the function $ \psi_Y = \sum_j \theta_j^{2} |f_j|^2 $ and for each $ t > 0 $, let $ \psi_t = \log(\psi_Y + t^2) $ and $ \chi_t = \chi + \delta  \sqrt{-1}  \partial \bar{\partial} \psi_t $. Also as in \cite[Lemma~2.1]{dempaun}, given $ \varepsilon > 0 $, by choosing $ \delta $ small enough, we can ensure that $ \frac{\chi_t^n}{\chi^{n}} > 1 - \varepsilon $ for all $ t $ small enough. \\
		
		 Now consider the family of differential equations for $ \Omega_t \in [\Omega_0 + t \eta] $: 
		$$ \label{eqfamily} P^{n}_\theta(\Omega_t, \chi) = \chi_t^{n} - \chi^n + A_t \chi^n $$where $ A_t $ is a real number fixed by integrating both sides over $ M $. As the integral of the left-hand side is non-negative and as $ \chi_t \in [\chi_0] $, $ A_t $ is non-negative as well. It can also be seen that $ A_t $ is increasing in $ t $. Writing the right-hand side as $ f_t \chi^n $ where $ f_t = \frac{\chi_t^n}{\chi^{n}} - 1 + A_t $, we see that $ f_t > -\varepsilon + A_t > -\varepsilon $ and that $ \int_M f_t \chi^n  = A_t \int_M \chi^{n} \geq 0 $. By Proposition \ref{prop:gctwisted} and Theorem \ref{mainpdethm}, this equation can be solved for each $ t > 0 $ with the solutions $ \Omega_t $ satisfying the cone condition by choosing $ \varepsilon_{n, \theta} $ small enough. \\
		
	 As $ \int_M \Omega_t \wedge \chi^{n-1} $ is bounded independently of $ t $ (for $ t < 1  $, say), the family of solutions $ \Omega_t $ is weakly bounded and hence a subsequence $ {\Omega_{t_i}}_i $ converges weakly to a positive current $ \Theta \in [\Omega_0] $. It can be seen that $ \Theta $ satisfies the cone condition in the sense of definition \ref{def:conecurrent} on $ M $. By Lemma \ref{lem:bound}, $ C(\Omega_t + \chi)^n \geq P^n_\theta(\Omega_t, \chi) = \chi_t^n - \chi^{n} + A_t \chi^n $ for some $ C $ independent of $ t $ and hence, $ C' (\Omega_t + 2\chi)^n \geq \chi_t^n $ for some other $ C' $ independent of $ t $. Now we can proceed as in \cite[Proposition~2.6]{dempaun} to show that $ \Theta \geq \beta [Y] $ for some $ \beta > 0 $, concluding the proof of the theorem.
	\end{proof}
	
\section{Proof of Theorem \ref{maindhymthm}}
	
	We will first sketch a proof of the following regularization and gluing lemma:
\begin{lem}
	\label{regandglue}
	Let $ (M, \alpha) $ be a compact K\"ahler manifold. Suppose $ \Omega_0, \chi $ are smooth real $ (1, 1) $-forms on $ M $ such that there exists a K\"ahler  current $ T \in [\Omega_0] $ such that $ T \geq \beta [Y] $ for some $ \beta > 0 $ and a codimension-1 subvariety $ Y $. Suppose also that there exists an $ \epsilon > 0 $ such that $ T - 3 \epsilon \alpha $ satisfies the cone condition $ \overline{C^{m}_{\chi, \theta}} $ in the sense of definition \ref{def:conecurrent} on $ Y^c $. There exits a $ c > 0$ depending only on $ \alpha, \epsilon, \beta $ such that if there is a neighborhood $ U $ of  $ S = E_c(T) \cup Y$ on which there exists a smooth $ \Omega_U \in [\Omega_0] $ satisfying the cone condition, then there exists a smooth $ \Omega \in [\Omega_0] $ satisfying the cone condition on $ M $.
\end{lem}
\begin{proof}
	As $ T \in [\Omega_0] $, there exists a quasi-psh function $ \phi_T$ such that $ T = \Omega_0 +  \sqrt{-1} \partial \bar{\partial} \phi_T $. Similarly, there exists a smooth function $ \phi_U $ on $ U $ such that $ \Omega_U = \Omega_0 +  \sqrt{-1}  \partial \bar{\partial} \phi_U $ on $ U $. Cover $ M $ by a finite number of coordinate balls $ \{B^i_{6r}(x_i)\}_i $ centered at points $ x_i \in M$ such that $ \{ B^i_r(x_i) \}_i $ also cover $ M $ and such that on each $ B^i_{6r}(x_i) $, the coordinate distance is within $ 1 + \frac{1}{100} $ of the distance defined using the K\"ahler metric $ \alpha $ and such that on $ B^i_{6r}(x_i) $ 
	
	\begin{equation*}	
	\begin{split}
	& \alpha =  \sqrt{-1} \partial \bar{\partial} \phi_{\alpha}^i  \\
	& |\phi_{\alpha}^{i} - |z|^{2}| < \frac{\epsilon}{1000} r^{2}.
	\end{split}	
	\end{equation*}Such an $ r $ exists as $ \alpha $ is K\"ahler and $ M $ is compact. We will also require that for each $ y \in Y $, the $ \alpha $-ball of radius $ 2r $, $ B_{2r}(y) $, is contained in some $ B^i_{6r}(x_i) $, but this is automatically satisfied as each $ y \in Y $ is in $ B^i_r(x_i) $ for some $ i $ and hence, $B_{2r}(y) \subset B_{6r}^i(x_i) $ (as the Euclidean and $ \alpha $-distances on $ B^i_{6r}(x_i) $ are nearly equal) . \\

	As the $ B^i $s are convex, we have $ \displaystyle \Omega_0 =  \sqrt{-1}  \partial \bar{\partial} \phi^i_0 $ and $ \displaystyle \chi =  \sqrt{-1}  \partial \bar{\partial} \phi_{\chi}^i  $ on $ B^i_{6r}(x_i) $ for some smooth functions $ \phi^i_0,  \phi_{\chi}^i $. \\

	There exists some $ C > 0 $ such that $ 10 C \chi \leq \alpha $ on $ M $. Define $ \epsilon_1 = \epsilon C $. Let $ \delta_0 < r/4 $ be so small that each $ B^i_{11r/2}(x_i) $ can be covered by coordinate balls of radius $ \delta_0  $ and on each ball $ B^i_{2 \delta_0}(x) $ with $ x \in B^i_{11r/2}(x_i) $, there exists a constant coefficient form $ \chi_0 $ which satisfies
	\begin{equation*}
		0 < \chi_0 \leq \chi \leq \chi_0(1 + \epsilon_2)
	\end{equation*}where $ \epsilon_2 $ is chosen small enough so that $ \displaystyle  \frac{\epsilon_2 |\cot(\theta)|}{\epsilon_2 + 1} < \epsilon_1 $ and if $ \omega \in \Gamma_{\chi_0, \theta} $, then $ \omega + \epsilon_1 \chi \in \Gamma_{\chi, \theta} $ (by Lemma \ref{lem:perturb1}, $ \epsilon_2 $ can be chosen to depend only on $ \epsilon_1 $, $ \theta $, $ \text{dim}(M) $). \\
	
	Let $ 0 < \delta < \delta_0   $. If $ x \in B_{11r/2}^i(x_i) $ for some $ i $ with and $ B^i_\delta(x) \cap Y = \phi $, then by the definition of the cone condition of currents, 
	\begin{equation*}
	 \sqrt{-1} \partial \bar{\partial}(\phi^i_0 - 3 \epsilon \phi_{\alpha}^{i} + \phi_T)_{,\delta} \in C_{\chi_0(x), \theta}
	\end{equation*}where $ \phi_{,\delta} $ denotes the mollification at level $ \delta $. \\

	So we have $ \displaystyle \sqrt{-1} \partial \bar{\partial}(\phi^i_0 - 3 \epsilon \phi_{\alpha}^{i} + \phi_T)_{,\delta} + \chi_0(x)\cot(\theta) \in \Gamma_{\chi_0(x), \theta}$. By the choice of $ \epsilon_2 $ and Lemma \ref{lem:perturb1}, we get $ \displaystyle \sqrt{-1} \partial \bar{\partial}(\phi^i_0 - 3 \epsilon \phi_{\alpha}^{i} + \phi_T)_{,\delta} + \chi_0(x)\cot(\theta) + \epsilon_1 \chi(x) \in \Gamma_{\chi, \Theta} $. \\
	
	Hence, 
	\begin{equation}
		\sqrt{-1} \partial \bar{\partial}(\phi^i_0 - 3 \epsilon \phi_{\alpha}^{i} + \phi_T)_{,\delta} + \chi_0(x)\cot(\theta) + \epsilon_2 \chi(x) - \chi(x)\cot(\theta) \in C_{\chi(x), \theta}
	\end{equation}
	
	Now we have 
	\begin{equation*}
	\begin{split}
	\chi(x) - 2  \sqrt{-1}  \epsilon \partial \bar{\partial} \phi_{{\chi, \delta}}^{i} + \cot(\theta)(\chi_0(x) - \chi(x)) < & \chi(x)(\epsilon_1 + |\cot(\theta)|\frac{\epsilon_2}{1 + \epsilon_2}) - 2 \epsilon \alpha_{,\delta}(x) \\
	 < & 2(\chi(x) \epsilon_1 - \chi_{, \delta}(x) 10 C \epsilon) \\
	 = & 2 C \epsilon ( \chi(x)   -  10 \chi_{, \delta}(x) )
	\end{split}
	\end{equation*}
	By choosing $ \delta $ small enough, we see that this last expression can be made negative. Hence, we have
	\begin{equation*}
		\sqrt{-1} \partial \bar{\partial}(\phi^i_0 - \epsilon \phi_{\alpha}^{i} + \phi_T)_{, \delta}(x) \in C_{\chi(x), \theta}.
	\end{equation*}
	
	Define the psh function $ \phi_i := \phi^i_0 - \epsilon \phi_{\alpha}^{i} + \phi_T $ on $ B^i_{6r}(x_i) $. For $ x \in B^i_{11r/2}(x_i) $ and $ \delta < r/4 $, if we define the Lelong number at level $ \delta $ of $ \phi_i $ as
	\begin{equation*}
		\nu^i(x, \delta) = \frac{\hat{\phi}^i_{r/4}(x) - \hat{\phi}^i_{\delta}(x)}{\log(r/4) - \log(\delta)},
	\end{equation*} then $ \displaystyle \nu(T - \epsilon \chi, x) = \nu^i(x) = \lim_{\delta \rightarrow 0} \nu^i(x, \delta) $, $ \nu^i(x, \delta) $ is non-decreasing in $ x $ and as in \cite{gao}, it is seen to satisfy the following inequalities 
	\begin{equation*}
		\begin{split}
		0 \leq \hat{\phi}^i_{\delta}(x) - \hat{\phi}^i_{\delta/a}(x) \leq \nu^i(x, \delta)  \log(a) & \hspace{2em} \forall a \geq 1,  \\
		0 \leq \hat{\phi}^i_{\delta}(x) - {\phi}^i_{,\delta}(x) \leq \nu^i(x, \delta) a_n, &  
		\end{split}
	\end{equation*}where $ a_n $ is a constant depending only on the dimension $ n $ of $ M $ and the mollifier $ \rho $. \\
	
	We now define $ \displaystyle c = \min\bigg(\frac{ \epsilon r^2}{1000 a_n}, \frac{\beta}{1000}\bigg) $. Note that $ c $ only depends on $ r, \beta, \epsilon $ and that $ r $ in turn depends only on $ \alpha $. \\

	By taking $ \delta $ sufficiently small, it can be shown as in \cite[Section~4]{gao} that the regularized maximum of $ (\phi^i_0 - \epsilon \phi_{\alpha}^{i} + \phi_T)_{, \delta}(x) - \phi^i_0(x) $ on $ B^i_{3r}(x_i) $ and $ \phi_U + 3c \log(\delta) $ on an open neighborhood $ O \subset U $ of $ E_c(T) \cup Y $ is a smooth function $ \phi $ on $ M $ such that $ \Omega = \Omega_0 +  \sqrt{-1} \partial \bar{\partial} \phi $ satisfies the cone condition $ C_{\chi, \theta} $ away from $ Y  $. Further, by \cite[Proposition~4.1]{datpin20}, it can be ensured (by taking $ \delta $ smaller if necessary) that $ \Omega $ satisfies the cone condition $ C_{\chi, \theta} $ on all of $ M $.
	
\end{proof}
	
	We now prove Theorem \ref{maindhymthm} by induction on $ n = \dim(M) $. For $ n = 1 $, the theorem is trivial. Assume that $ n > 1 $ and the theorem has been proved in all dimensions $ < n $.	Let $ \eta $ be the curvature form of an ample line bundle $ L $ on $ M $ and let $ V $ be an $ m $-dimensional subvariety. Then for any $ (m-k, m -k) $-form $ F $, $ \displaystyle \int_V F \wedge \eta^{k}  = \int_{V \cap H_1 \cap \dots \cap H_k} F$ where $ H_i $ are some members of the linear system $ |L| $. So if we define $ \Omega_t := \Omega_0 + t \eta $, then by the hypothesis of Theorem \ref{maindhymthm},
	\begin{equation}
		\int_V P_\theta^m(\Omega_t, \chi) = \sum_r \int_V \binom{m}{r} P_\theta^{m-r}(\Omega, \chi) \wedge t^r \eta^r \geq \int_V t^m \eta^m  
	\end{equation}Hence, for each $ t > 0 $, the pair $ (\Omega_t, \chi) $ satisfies the uniform numerical condition in \cite[Proposition~5.2]{gao}. By the same proposition, there exist metrics $ {}_1\Omega_t \in [\Omega_t] $ satisfying the cone condition $ C_{\chi, \theta} $ on $ M $. Let $ s $ be a holomorphic section of $ L $ with non-trivial zero section $ Y $. Let $ \chi_Y $ be a K\"ahler metric cohomologous to the current $ [Y] $. By Theorem \ref{concofmass}, we obtain a positive current $ \Theta \in [\Omega_0] $ satisfying the cone condition $ C_{\chi, \theta} $ on $ M $ in the sense of definition \ref{def:conecurrent} and also satisfying $ \Theta \geq 2 \beta [Y] $ for some $ \beta > 0 $. Let $ T = \Theta -\beta[Y] + \beta \chi_Y $. Then $ T \in [\Omega_0] $ and $ T \geq \beta [Y] + \beta \chi_Y $ and $ T = \Theta + \beta \chi_Y $ outside $ Y $. Thus, by taking $ \epsilon $ sufficiently small, we see that $ T - \epsilon \chi $ is a K\"ahler current satisfying the cone condition outside $ Y $. We will now be done by Theorem \ref{regandglue} provided we prove the following result:
	
	\begin{prop}
		\label{mainnbdcone}
		Let $ Z $ be a proper subvariety of $ M $, then there exists an open neighborhood $ U $ of $ Z $ and a smooth function $ \phi_U $ on $ U $ such that $ \Omega_U := \Omega_0 +  \sqrt{-1} \partial \bar{\partial} \phi_U $ satisfies the cone condition $ C_{\chi, \theta} $ on $ U $
	\end{prop} 
	
	We prove this by induction on $ m = \text{dim}(Z) $. If $ m = 0 $, this is trivial. Assume $ 0 < m < n $ and that the proposition has been proven for all subvarieties of dimension $ < m $. If $ Z $ is smooth, then by Theorem \ref{maindhymthm} applied to the manifold $ Z $, we get a smooth function $ \phi_Z $ on $ Z $ such that $ \Omega = \Omega_0|_Z +  \sqrt{-1} \partial \bar{\partial} \phi_Z $ satisfies the cone condition $ C_{\chi|_Z, \theta} $ on $ Z $. We define, on a tubular neighborhood $ U $ of $ Z $, the form $ \Omega_U = \Omega_0 +  \sqrt{-1} \partial \bar{\partial} (\phi_Z(p(x)) + C d^2_Z(x)) $, where $ p $ is the projection from $ U  $ to $ Z $, $ d_Z $ is the distance from $ Z $ as measured by a K\"ahler form ($ d_Z^2 $ is smooth for $ U $ sufficiently small) and $ C $ is a positive constant. By taking $ C $ sufficiently large, it can be seen that $ \Omega_U $ satisfies the cone condition $ C_{\chi, \theta}$ on $ U $. Hence we are done if $ Z $ is smooth. \\
	
	If $ Z $ is singular, we obtain a manifold $ \tilde{M} $ by a finite sequence of blow-ups along the components of $ Z_{sing} $ such that the proper transform $ \tilde{Z} $ of $ Z $ is smooth. Let $ \pi: \tilde{M} \rightarrow M $ be the projection. $ \pi $ is the identity outside $ E = \pi^{-1}(Z_{sing}) $ and $ E $ is the union of the (pullbacks of the) exceptional divisors $ E_1, \dots, E_k $. As in \cite[Chapter VII, Proposition 12.4]{dembook}, for small enough $ c > 0 $,  $ \displaystyle \alpha := \pi^{*}\chi + c(F_1 + \dots + F_k) > 0 $ on $ \tilde{M} $, where $ F_i $ is the curvature of the line bundle $ [-E_i] $ with respect to a metric $ h_i $ (pulled back to $ \tilde{M} $). Also, we note that if $ s_i $ is the defining section of $ E_i $, then outside $ E $, we have 
	\begin{equation}
		\alpha = \pi^* \chi + c \sqrt{-1} \partial \bar{\partial} \log |s_E|^2
	\end{equation}where $ |s_E|^{2} := \prod_i |s_i|^{2}_{h_i} $.
	\\  
	
	 Now with $ \eta $ as before, for each $ t > 0 $, we find $ {}_{1}\Omega_t \in [\Omega_t] $ such that $ P_\theta^k({}_{1}\Omega_t, \chi) > 0 $ on $ M $ for $ k = 0, \dots, n $. By continuity, there exist constants $ \epsilon_t > 0 $ such that $ P_\theta^k({}_{1}\Omega_t, \chi) > \epsilon_t \cdot {}_{1}\Omega_t^k $ on $ M $ for $ k = 0, \dots, n $. \\
	
	We now want to pull back these forms to obtain forms on $ \tilde{M} $ satisfying the cone condition. However, as $ \pi $ is only an isomorphism outside $ E $, the pullback of a K\"ahler form will only be non-negative. To overcome this issue, we will add small amounts of the K\"ahler metric $ \alpha $ to the pulled back forms. To do this while still maintaining the cone conditions, we need the next proposition:
	
	\begin{prop}
		\label{degentonon}
		Let $ 0 < \epsilon < 1 $, $ 0 < a < 1 $ and $ A > 0 $. Suppose $ \Omega $, $ \chi $ are positive $ (1,1) $-forms on an $ n $-dimensional complex manifold $ X $ satisfying $ P_\theta^k(\Omega, \chi) \geq \epsilon \Omega^k  $ for $ 1 \leq k \leq m \leq n $ at a point $ p \in X $. Let $ \alpha $ be a positive-definite $ (1, 1) $-form at $ p $ and define $ \Omega_{s, A} = \Omega + As \alpha $ and $ \chi_s = \chi + s^{N} \alpha $, where $ N > \text{dim}(X) = n $. As $ \alpha $ is positive definite at $ p $, there exists $ C_\chi > 0 $ such that $ \chi \leq C_\chi \alpha $ at $ p $. We claim that there is an $ s_0 = s_0(N, C_\chi, m, A, a, \theta) > 0 $ such that, at $ p $, $ P_\theta^k(\Omega_{s, A}, \chi_s) \geq a\epsilon {\Omega_{s, A}}^k $ for every $ s < s_0 $ and $ k \leq m $. Note that $ s_0 $ depends on $ p $ only through $ C_\chi $.
	\end{prop}
	\begin{proof}
		We use induction on $ m $. For $ m = 1 $, $ P^m_{\theta}(\Omega_{s, A}, \chi_{ s}) = \Omega_{s, A} \geq a \epsilon \Omega_{s, A} $ for any $ s > 0 $. Assume that the proposition has been proven for all $ m < \nu \leq n $. By the induction hypothesis, it suffices to show that $ \exists s_0 = s_0(N, C_\chi, \nu, A, a, \theta) $ such that $P^{\nu}_{\theta}(\Omega_{s, A}, \chi_s) \geq a\epsilon {\Omega_{s, A}}^{\nu} $. As $ P^{\nu}_{\theta}(\Omega_{s, A}, \chi_s) = P^{\nu}_{\theta}(\Omega_{s, \frac{A}{2}} + \frac{As\alpha}{2}, \chi_s) $, 
		\begin{alignat*}{2}
		 P^{\nu}_{\theta}(\Omega_{s, A}, \chi_s) = \sum_{r } \binom{\nu}{r} P^{\nu - r }_{\theta}(\Omega_{s, \frac{A}{2}}, \chi_s)\bigg(\frac{As\alpha}{2}\bigg)^{r} \\
		  = P^{\nu}_{\theta}(\Omega_{s, \frac{A}{2}}, \chi_s) +  \sum_{r > 0} \binom{\nu}{r} P^{\nu - r }_{\theta}(\Omega_{s, \frac{A}{2}}, \chi_s)\bigg(\frac{As\alpha}{2}\bigg)^{r}
		\end{alignat*}
		By the induction hypothesis, there is an $ s_1 = s_1(N, C_\chi, \nu - 1, \frac{A}{2}, \frac{a + 1}{2}, \theta) > 0 $ such that if $ s < s_1  $ and $ r \leq \nu - 1 $, then 
		$$ P^{r }_{\theta}(\Omega_{s, \frac{A}{2}}, \chi_s) \geq \epsilon \bigg(\frac{a + 1}{2} \bigg)(\Omega_{s, \frac{A}{2}})^k $$. So 
		$$ P^{\nu}_{\theta}(\Omega_{s, A}, \chi_s) \geq P^{\nu}_{\theta}(\Omega_{s, \frac{A}{2}}, \chi_s) + \sum_{r > 0} \binom{\nu}{r} \epsilon \frac{a + 1}{2}(\Omega_{s, \frac{A}{2}})^{\nu - r} \bigg(\frac{As\alpha}{2}\bigg)^{r}  $$
		\\
		
		To bound the first term on the right, we write $$ P^{\nu}_{\theta}(\Omega_{s, \frac{A}{2}}, \chi_s) = P^{\nu}_{\theta}(\Omega_{s, \frac{A}{2}}, \chi) + \sum_{p > 1, q} c_{p, q} (\Omega_{s, \frac{A}{2}})^{\nu - p - q} (s^N \alpha)^{p} \chi^q  $$where $ c_{p, q} $ depend only on $ \nu $ and $ \theta $. If $ c_{p, q} \geq 0$ for some pair $ (p, q) $, then $ \displaystyle c_{p, q} (\Omega_{s, \frac{A}{2}})^{\nu - p - q} (s^N \alpha)^{p} \chi^q \geq 0 $. If $ c_{p, q} < 0 $, then $ \displaystyle c_{p, q} (\Omega_{s, \frac{A}{2}})^{\nu - p - q} (s^N \alpha)^{p} \chi^q \geq c_{p, q} C_\chi^{q} (\Omega_{s, \frac{A}{2}})^{\nu - p - q} (s^N)^p \alpha^{p + q} \chi $ so that 
		$$ \sum_{p > 1, q} c_{p, q} (\Omega_{s, \frac{A}{2}})^{\nu - p - q} (s^N \alpha)^{p} \chi^q \geq \sum_{r > 0} s^N d_r(s) (\Omega_{s, \frac{A}{2}})^{\nu - r} \alpha^{r} $$ where $ d_r(s) $ is a polynomial in $ s $ with coefficients depending on the $ c_{p, q} $'s and $ C_\chi $ only. Thus, we can find a constant $ s_2 = s_2(N, C_\chi, m, A, a, \theta ) $ such that for all $ s < s_2 $, $ \displaystyle d_r(s) + \epsilon \binom{\nu}{r} \bigg( \frac{a + 1}{2}\bigg) \bigg(\frac{As}{2} \bigg)^r \geq \epsilon a \binom{\nu}{r} \bigg(\frac{As}{2} \bigg)^r  $. So we get
		$$ P^{\nu}_{\theta}(\Omega_{s, A}, \chi_s) \geq P^{\nu}_{\theta}(\Omega_{s, \frac{A}{2}}, \chi) + \sum_{r > 0} \epsilon a \binom{\nu}{r} (\Omega_{s, \frac{A}{2}})^{\nu - r} \bigg(\frac{As\alpha}{2}\bigg)^{r} $$for all $ s < s_0 := \min(s_1, s_2) $. 
		\\
		
		Lastly, for any $ s > 0 $,
		\begin{alignat*}{2}
		P^{\nu}_{\theta}(\Omega_{s, \frac{A}{2}}, \chi) = \sum_{r } \binom{\nu}{r} P^{\nu - r }_{\theta}(\Omega, \chi) \bigg(\frac{As\alpha}{2}\bigg)^{r} \\
		\geq \sum_{r } \epsilon \binom{\nu}{r} \Omega^{\nu -r} \bigg(\frac{As\alpha}{2}\bigg)^{r} \\
		= \epsilon (\Omega_{s, \frac{A}{2}})^{\nu}
		\end{alignat*}So for any $ s < s_0 $, we have 
		$$ P^{\nu}_{\theta}(\Omega_{s, A}, \chi_s) \geq a \epsilon (\Omega_{s, A})^{\nu}  $$thus completing the proof
		
	\end{proof}

	Hence, for each $ s > 0 $, if we define $ \Omega_{t, s} := \pi^* {}_{1}\Omega_t + s \alpha $ and $ \chi_{ s} = \pi^* \chi + s^{N} \alpha $, we see by the above proposition, there exist constants $ \bar{s}_t \in (0, 1) $ such that if $ s_t = \bar{s}_t(1 - e^{-t} )  $, then $ P_\theta^k(\Omega_{t, s_t}, \chi_{s_t}) \geq \frac{\epsilon_t}{2} \cdot \Omega_{t, s_t}^k $ for $ k = 1, \dots, n - 1 $.
	\\
	
	From here on, the restriction of a differential form on $ \tilde{M} $ to $ \tilde{Z} $ will be denoted by the same symbol. We will produce a metric satisfying the cone condition on $ \tilde{Z} $ by applying Lemma \ref{regandglue} on $ \tilde{Z} $. Let $ L $ be a very ample line bundle on $ \tilde{Z} $ and let $ \tilde{Y} $ be the zero set of a non-trivial holomorphic section of $ L $. We now use a concentration of mass similar to Theorem \ref{concofmass} to find a current $ \Theta \in [\pi^* \Omega_0] $ on $ \tilde{Z} $ such that $ \Theta \geq \beta [\tilde{Y}] $ and $ \Theta $ is the weak limit of metrics satisfying the cone condition. \\
	
	 Let $ u > 0 $ be a positive constant and for each $ t > 0 $, as in the proof of theorem \ref{concofmass} , define $ \phi_{\tilde{Y}} $ and $ \chi_{u, t} := \chi_{u} +  \sqrt{-1} \delta \partial \bar{\partial} \log(\phi^2_{\tilde{Y}}(x) + t^2) $ where $ \delta $ is chosen small enough that $ \frac{\chi_{u, t}^m}{\chi_{u}^m} > 1 - \frac{\varepsilon_{m, \theta}}{10} $ for all sufficiently small $ t $. Consider the following PDE for $ {}_{1} \Omega_{t, s_t} \in [ \Omega_{t, s_t} ] $
	\begin{equation}
	\label{eqfamilyZ}
		P^m_{\theta}({}_{1} \Omega_{t, s_t}, \chi_{s_t}) = \chi_{u, t}^m - \chi_{s_t}^{m} + B_t \chi_{s_t}^{m}  
	\end{equation}where $ B_t $ is a constant determined by integrating both sides of the equation over $ \tilde{Z} $. As $ \int_{\tilde{Z}} P^m_{\theta}(\pi^* \Omega_0, \pi^* \chi) > 0 $ and as $ \int_{\tilde{Z}} \chi_{u, t}^m = \int_{\tilde{Z}} \chi_{u}^m $, we see that by taking $ u $ small enough, we can have $ B_t > 0 $ for $ t << u $ and hence, 
	\begin{equation*}
		f_t := \frac{\chi_{u, t}^m}{\chi_{s_t}^{m}} - 1 + B_t > -\varepsilon_{m, \theta}
	\end{equation*} so that the equations \ref{eqfamilyZ} can be solved for each $ t > 0 $ small enough. As in Theorem \ref{concofmass}, we can take a sequence of $ t's $ approaching $ 0 $ such that the corresponding sequence of $ {}_{1} \Omega_{t, s_t} $'s converges weakly to a current $ \Theta \in [\pi^* \Omega_0] $ satisfying $ \Theta \geq 2\beta [\tilde{Y}] $ for some $ \beta > 0 $. 
	\\
	
	Let $ \chi_Y $ be a K\"ahler metric in the same cohomology class as $ [\tilde{Y}] $ and let $ 0 <  \epsilon_1 < 1 $ be such that $ \displaystyle \beta \chi_Y -  \epsilon_1 \chi_{ s} \geq \frac{\beta}{2}\chi_Y $ for all small $ s $, say $ s < 1 $. For each $ t > 0 $ and $ s > 0 $, set $$ {}_{2} \Omega_{t, s_t} = {}_{1} \Omega_{t, s_t} - \beta [\tilde{Y}] + \beta \chi_Y - 3 \epsilon_1 \chi_s $$  $$ {}_{3, s} \Omega_{t, s_t} = {}_{2} \Omega_{t, s_t} + s \alpha $$ and $$ T'_s := \lim_{t \rightarrow 0 } {}_{3} \Omega_{t, s_t} = \Theta - \beta [\tilde{Y}] + \beta \chi_Y + s \alpha -  \epsilon_1 \chi_s =: T_s -  \epsilon_1 \chi_{ s} $$($ \lim $ denotes the weak limit). By the choice of $ \epsilon_1 $ and as $ \Theta \geq \beta [\tilde{Y}] $, $ T'_s $ is a K\"ahler current for all $ s $. \\
	
	\begin{prop}
		\label{currentconeZ}
		There exists $ s_1 > 0 $ such that for all $ s < s_1 $, the pairs $ T'_s, \chi_{ s} $ satisfy the cone condition on $ \tilde{Z} \cap \tilde{Y}^{c} $ in the sense of definition \ref{def:conecurrent}. 
	\end{prop}
	\begin{proof}
		Outside $ \tilde{Y} $, we have $ \displaystyle {}_{2} \Omega_{t, s_t} \geq {}_{1} \Omega_{t, s_t} + \frac{\beta}{2}\chi_Y  $, so by proposition \ref{degentonon} and the fact that $ \chi_{s_t} $ can be bounded independently of $ t $ in terms of $ \chi_Y $, $ \exists \varepsilon > 0 $ such that $ P^k_{\theta}({}_{2} \Omega_{t, s_t}, \chi_{s_t}) \geq 2 \varepsilon \cdot {}_{2} \Omega_{t, s_t}^k  $ for $ k = 1, \dots, m $ on $ \tilde{Y}^c $. Hence, by proposition \ref{degentonon}, there exists $ s_0 > 0 $ independent of $ t $ such that $ \forall s < s_0 $, $  P^k_{\theta}({}_{3, s} \Omega_{t, s_t}, \chi_{s_t} + s^N \alpha) \geq \varepsilon \cdot {}_{3, s} \Omega_{t, s_t}^k $ on $ \tilde{Y}^{c} $ for $ k = 1, \dots, m $. In particular, $ P^k_{\theta}({}_{3, s} \Omega_{t, s_t}, \chi_{s_t} + s^N \alpha) > 0 $. As $ \chi_{ s} < \chi_{s_t} + s^N \alpha $, we also have $ P^k_{\theta}({}_{3, s} \Omega_{t, s_t}, \chi_{s} ) > 0 $ and hence, $ T'_s $, the weak limit of $ {}_{3, s} \Omega_{t, s_t}  $ also satisfies the cone condition outside $ \tilde{Y} $ with respect to the K\"ahler metric $ \chi_{ s} $.
	\end{proof}

As $ \displaystyle T_s - \epsilon_1  \chi_s  = T'_s \geq \Theta - \beta [\tilde{Y}] + s \alpha + \frac{\beta}{2} \chi_Y $ for all $ s < 1 $, the same reasoning as above shows that there exists $ \epsilon_2 > 0 $ such that $ \displaystyle T'_s - 3\epsilon_2 \alpha  \geq \Theta - \beta [\tilde{Y}] + s \alpha + \frac{\beta}{4} \chi_Y $ for all $ s < 1 $. By the proof of Proposition $ \ref{currentconeZ} $, there exists $ s_2 > 0 $ such that for all $ s < s_2 $, the currents $ T'_s - 3 \epsilon_2 \alpha $ satisfies the cone condition on $ \tilde{Z}
\cap \tilde{Y}^{c} $ with respect to $ \chi_s $.  \\

Since $ T'_s \geq \beta [\tilde{Y}] $ and the pairs $ (T'_s - 3 \epsilon_2 \alpha, \chi_s ) $ satisfy the cone condition on $ \tilde{Z} \cap \tilde{Y}^{c} $ (for $ s < s_2 $), we are in the position to apply the regularization Lemma \ref{regandglue} on the K\"ahler manifold $ (\tilde{Z}, \alpha) $. Define $ c > 0 $ to be the constant required in Lemma \ref{regandglue}. Note that $ c $ only depends on $ \alpha, \beta, \epsilon_2, dim(\tilde{Z}) $ and is hence independent of $ s $.  Let $ \tilde{S} $ be the set $ E_c(\Theta) \cup \tilde{Y} $. From the definition of $ T_s $, it follows that $ E_c(T_s) \subseteq \tilde{S} $ for all $ s $. Let $ S = \pi(\tilde{S}) $. By the induction hypothesis, there exists a neighborhood $ U $ of $ S $ and a $ \Omega_S \in [\Omega_0] $ on $ U $ satisfying the cone condition with respect to $ \chi $ on $ U $. By continuity, there exist constants $ \epsilon_3, \epsilon_4 > 0 $ such that $ P^k_{\theta}(\Omega_S - \epsilon_3 \chi, \chi ) \geq 2\epsilon_4 \Omega_S^{k} $ for $ k = 1, \dots, n - 1 $ on $ U $ (by shrinking $ U $ if necessary). Without loss of generality, we can assume that $ \epsilon_3 < \epsilon_1 $. By Proposition \ref{degentonon}, we see that for all sufficiently small $ s > 0 $, $ P^k_{\theta}(\Omega_S + s \alpha - \epsilon_3 \pi^* \chi, \chi_s ) \geq \epsilon_4 (\Omega_S + s \alpha - \epsilon_3 \pi^* \chi ) ^{k} $ on $ \pi^{-1}(U) $ and hence on a neighborhood of $ \tilde{S} $ in $ \tilde{Z} $ for $ k = 1, \dots, m $. Define $ \Omega_{S, s} = \pi^*\Omega_S + s \alpha $.

\begin{lem}
	\label{pullbackS}
	There exists an $ s_3 > 0 $ depending only on $ \theta $, $ m $, $ \epsilon_4 $ such that for all $ s < s_3 $ the pairs $ \Omega_{S, s} - \epsilon_3 \chi_{ s}, \chi_{ s} $ satisfy the cone condition on $ \pi^{-1}(U) $
\end{lem}
\begin{proof}
	\begin{alignat*}{2}
		P^k_{\theta}(\Omega_{S, s} - \epsilon_3 \chi_s, \chi_s ) = \sum_{r} \binom{k}{r} P^{k-r}_{\theta}(\Omega_{S,s} - \epsilon_3 \pi^* \chi, \chi_s ) (-\epsilon_3 s^N)^r \alpha^r \\
		= P^k(\Omega_{S,s} - \epsilon_3 \pi^* \chi, \chi_s) + \sum_{r > 0} \binom{k}{r} P^{k-r}_{\theta}(\Omega_{S,s} - \epsilon_3 \pi^* \chi, \chi_s ) (-\epsilon_3 s^N)^r \alpha^r \\
		\geq \epsilon_4 (s \alpha)^k + \sum_{r > 0} \binom{k}{r} P^{k-r}_{\theta}(\Omega_{S,s} - \epsilon_3 \pi^* \chi, \chi_s ) (-\epsilon_3 s^N)^r \alpha^r
	\end{alignat*} 
	This last sum is a polynomial in $ s $ divisible by $ s^N $ with coefficients which can be bounded by exterior powers $ \alpha $ independent of $ s $ and $ \epsilon_3 $ (we can assume both these numbers are $ < $ 1). Hence, by taking $ s $ small enough, we see that $P^k_{\theta}(\Omega_{S, s} - \epsilon_3 \chi_s, \chi_s ) > 0$
\end{proof}	
	
	Similarly, we also obtain a constant $ \epsilon_5 > 0 $ and family of smooth metrics $ \Omega_{E, s} $ on neighborhood $ U_E $ of $ E $ in $ \tilde{M} $ such that $ (\Omega_{E, s} - \epsilon_5 \chi_{ s}, \chi_s) $ satisfy the cone condition on $ U_E $ for all $ s < s_4 $. 
	\\
	
	For the remainder of the paper, fix $ \displaystyle s = \frac{\epsilon_6}{2}$, where $ \displaystyle \epsilon_6 := \frac{\min( \epsilon_1, \epsilon_2 \epsilon_3, \epsilon_5, s_1, s_2, s_3, s_4)}{1000} $. As $ \epsilon_1 > \epsilon_3 $, $ T: = T_s - \epsilon_3 \chi_s > T_s - \epsilon_1 \chi_s $ and hence, $ T - 3 \epsilon_2 \alpha  $ satisfies the cone condition on $ \tilde{Z} \cap \tilde{Y}^{c} $. We can therefore use Lemma \ref{regandglue} to glue the regularizations of $ T  $ and $ \Omega_{S,s} - \epsilon_3 \chi_s $ on the K\"ahler manifold $ (\tilde{Z}, \alpha) $ to obtain a metric $ \Omega_{\tilde{Z}, s} \in [\pi^{*}\Omega_0 + s\alpha - \epsilon_3 \chi_s] $ satisfying the cone condition $ C_{\chi_s, \theta}^{m} $ on $ \tilde{Z} $. We extend $ \Omega_{\tilde{Z}, s} $ to a metric $ \displaystyle \Omega_{U,s} = \pi^{*}\Omega_0 + s\alpha - \epsilon_3 \chi_s +  \sqrt{-1} \partial \bar{\partial} \phi_U $ on a neighborhood of $ \tilde{Z} $ such that the $ \Omega_{U,s} $ satisfies $ C_{\chi_s, \theta} $ on $ U $. \\
	
	Suppose $ \displaystyle \Omega_{E, s} = \pi^* \Omega_0 + s \alpha +  \sqrt{-1} \partial \bar{\partial} \phi_E $ on $ U_E $. Let $ C_1 > 0 $ be large enough as to imply $ \displaystyle (\epsilon_3 - \epsilon_6) \chi_s >  \sqrt{-1} \frac{\partial \bar{\partial} \log |s_E|}{C_1} $ on $ E^c $ (on $ E^c $, $ \displaystyle  \sqrt{-1} \partial \bar{\partial} \log|s_E|^2 = F$ for some smooth and bounded form $ F $, so such a $ C_1 $ exists), then $ \displaystyle \pi^* \Omega_0 + s \alpha - \epsilon_6 \chi_s +  \sqrt{-1} \partial \bar{\partial} (\phi_U + \frac{\log |s_E|}{C_1} ) $ satisfies the cone condition on $ U \cap E^c $. If we define $ \displaystyle \phi_O := \tilde{\max}(\phi_E, \phi_U + \frac{ \log|s_E|}{C_1} + C_2) $ for some large constant $ C_2 $, then $ \phi_O $ will be smooth on a neighborhood $ O $ of $ \tilde{Z} \cup E $ ($ \phi_O $ coincides with $ \phi_E $ near $ E $) and by the choice of $C_1$, $ \displaystyle \Omega_O = \pi^* \Omega_0 + s \alpha - \epsilon_6 \chi_s +  \sqrt{-1} \partial \bar{\partial} \phi_O $ satisfies the cone condition on $ O $. \\
	
	Next, we let $$ \phi_f = \phi_O + (s  - \epsilon_6 s^N)c \log |s_E|^{2}. $$ As $ \alpha = \pi^{*} \chi + c  \sqrt{-1} \partial \bar{\partial} \log |s_E|^{2} $ outside $ E $ and as $ s < \epsilon_6  $, we see that $ \pi^*\Omega_0 +  \sqrt{-1} \partial \bar{\partial} \phi_f  $ satisfies the cone condition $ C_{\chi_s, \theta} $ and hence $ C_{\pi^* \chi, \theta} $ outside $ E $. The push-forward $ \phi $ of $ \phi_f $ is then such that $ \displaystyle \Omega_1 := \Omega_0 +  \sqrt{-1} \partial \bar{\partial} \phi $ satisfies the cone condition $ C_{\chi, \theta} $ on a neighborhood of $ Z \cap Z_{sing}^c $. As $ dim(Z_{sing}) < dim(Z) $, we can obtain a metric satisfying the cone condition in a neighborhood of $ Z_{sing} $ by the induction hypothesis. This metric glued with $ \Omega_1 $ then gives an $ \Omega \in [\Omega_0] $ satisfying the cone condition $ C_{\chi, \theta} $ in a neighborhood of $ Z $ (as $ \phi $ has a logarithmic singularity near $ Z_{sing} $, the regularized maximum of $ \phi $ and the other potential will be equal to that potential in a neighborhood of $ Z_{sing} $), thus completing the proof of Proposition \ref{mainnbdcone} and hence Theorem \ref{maindhymthm}.

\printbibliography

\end{document}